%% file: mainArxiv.tex
\documentclass[paper=a4,abstract=true,numbers=noenddot]{scrartcl}

\usepackage[english]{babel}  
\usepackage{hyperref}
\usepackage[intlimits]{amsmath}
\usepackage{amsthm}
\usepackage{amssymb} 
\usepackage{units} 
\usepackage{amsfonts,mathrsfs,bbm}
\allowdisplaybreaks[1]
\usepackage[numbers]{natbib}
\usepackage{graphicx}
\usepackage[space]{grffile} 
\usepackage{tikz}
\usepackage{pgfplots}
\usepackage[T1]{fontenc}
\usepackage{mathptmx}
\usepackage[scaled=.92]{helvet}
\usepackage{courier}

\usepackage[a4paper,left=2.5cm,right=2.5cm,top=3cm,bottom=2.5cm]{geometry}
\usepackage{authblk}
\usepackage{amsmath,bm}
\usepackage{graphicx}
\usepackage{mathrsfs}
\usepackage{amssymb}
\usepackage{physics}
\usepackage{tensor}
\usepackage{xcolor}
\usepackage{enumerate}

\theoremstyle{remark}
\newtheorem*{remark}{\textbf{Remark}}

\title{On the space-time discretization of variational retarded potential
  boundary integral equations}

\author[1]{D. P\"olz}
\author[1]{M. Schanz}
\affil[1]{Institute of Applied Mechanics, Graz University of
  Technology, Technikerstraße 4/II, 8010 Graz, Austria,
  dominik.poelz@gmail.com, m.schanz@tugraz.at}
\date{}                     
\input{macros}

\begin{document}
	
\maketitle
\section*{Abstract}
  This paper discusses the practical development of space-time boundary element
  methods for the wave equation in three spatial dimensions. The employed trial
  spaces stem from simplex meshes of the lateral boundary of the space-time
  cylinder. This approach conforms genuinely to the distinguished structure of
  the solution operators of the wave equation, so-called retarded potentials.
  Since the numerical evaluation of the arising integrals is intricate, the bulk
  of this work is constituted by ideas about quadrature techniques for retarded
  layer potentials and associated energetic bilinear forms. Finally, we
  glimpse at algorithmic aspects regarding the efficient implementation of
  retarded potentials in the space-time setting. The proposed methods are
  verified by means of numerical experiments, which illustrate their capacity.

\textbf{Keywords}: wave equation; boundary element method; Bubnov-Galerkin; light cone

\input{content}

\bibliography{references}
\bibliographystyle{ieeetr}

\end{document}

%% file: macros.tex
\usepackage{amsmath}
\usepackage{mathtools}
\usepackage{amssymb}
\usepackage{amsfonts}
\usepackage{amsthm}
\usepackage{caption}
\usepackage[labelformat=brace]{subcaption}
\usepackage{xargs}
\usepackage{ifthen}
\usepackage{booktabs}
\usepackage[capitalize]{cleveref} 
\usepackage{algpseudocode,algorithm,algorithmicx}
\usepackage{siunitx}

\newtheorem{theorem}{Theorem}
\newtheorem{lemma}[theorem]{Lemma}

\newtheorem{definition}[theorem]{Definition}

\graphicspath{{pictures/}}

\crefname{appendix}{}{}
\Crefname{appendix}{}{}

\crefformat{equation}{\textup{#2(#1)#3}}
\crefrangeformat{equation}{\textup{#3(#1)#4--#5(#2)#6}}
\crefmultiformat{equation}{\textup{#2(#1)#3}}{ and \textup{#2(#1)#3}}
{, \textup{#2(#1)#3}}{, and \textup{#2(#1)#3}}
\crefrangemultiformat{equation}{\textup{#3(#1)#4--#5(#2)#6}}%
{ and \textup{#3(#1)#4--#5(#2)#6}}{, \textup{#3(#1)#4--#5(#2)#6}}{, and %
  \textup{#3(#1)#4--#5(#2)#6}}

\Crefformat{equation}{#2Equation~\textup{(#1)}#3}
\Crefrangeformat{equation}{Equations~\textup{#3(#1)#4--#5(#2)#6}}
\Crefmultiformat{equation}{Equations~\textup{#2(#1)#3}}{ and \textup{#2(#1)#3}}
{, \textup{#2(#1)#3}}{, and \textup{#2(#1)#3}}
\Crefrangemultiformat{equation}{Equations~\textup{#3(#1)#4--#5(#2)#6}}%
{ and \textup{#3(#1)#4--#5(#2)#6}}{, \textup{#3(#1)#4--#5(#2)#6}}{, and %
  \textup{#3(#1)#4--#5(#2)#6}}

\newcommand{\ra}[1]{\renewcommand{\arraystretch}{#1}}

\DeclareMathOperator{\sgn}{sgn}
\DeclareMathOperator{\diam}{diam}
\DeclareMathOperator{\dist}{dist}
\DeclareMathOperator{\potSl}{S}
\DeclareMathOperator{\potDl}{D}
\DeclareMathOperator{\bioSl}{V}
\DeclareMathOperator{\bioDl}{K}
\DeclareMathOperator{\opId}{I}
\DeclareMathOperator{\opT}{T}
\DeclareMathOperator{\opN}{N}

\DeclareMathOperator{\opQ}{Q}
\DeclareMathOperator{\opA}{A}
\DeclareMathOperator{\sons}{successors}
\DeclareMathOperator{\leaves}{leaves}

\newcommand*{\TX}[1]{\mathbf{#1}}
\newcommand*{\eIp}[2]{\left\langle#1,#2\right\rangle}
\newcommand*{\gIp}[3]{\left(#1,#2\right)_{#3}}
\newcommandx*{\NN}[1][1={}]{\mathbb{N}^{#1}}
\newcommandx*{\RR}[1][1={}]{\mathbb{R}^{#1}}
\newcommandx*{\CC}[1][1={}]{\mathbb{C}^{#1}}
\newcommandx*{\ZZ}[1][1={}]{\mathbb{Z}^{#1}}
\newcommand*{\opTk}[1]{\opT_{#1}}
\newcommand*{\opNk}[1]{\opN_{#1}}
\newcommand*{\intOp}[3]{\int_{#1} #2 \mathrm{d}{#3}}
\newcommandx*{\nSphere}{\mathbb{S}}
\newcommandx*{\sDist}[1][1={}]{\dist_{#1}}
\newcommand*{\fSol}{\mathcal{G}}
\newcommandx*{\trDiri}[2][1={},2={}]{
\ifthenelse{\equal{#2}{}}{\gamma_{0}^{#1}}{\gamma_{0,#2}^{#1}}  
}
\newcommandx*{\trNeum}[2][1={},2={}]{
\ifthenelse{\equal{#2}{}}{\gamma_{1}^{#1}}{\gamma_{1,#2}^{#1}}
}
\newcommandx*{\spaceLp}[3][1={},2={},3={}]{
\ifthenelse{\equal{#2}{}}{L^{#1}_{#3}}{L^{#1}_{#3}(#2)}  
}
\newcommandx*{\spaceHs}[3][1={},2={},3={}]{
\ifthenelse{\equal{#2}{}}{H^{#1}_{#3}}{H^{#1}_{#3}(#2)}  
}
\newcommand*{\LTwo}[1]{\spaceLp[2][#1][]}

\newcommand*{\LTSig}{\LTwo{\Sigma}}

\newcommand*{\intext}{\mathfrak{d}}
\newcommandx*{\rvertPnl}[1][1={}]{\mskip-6mu\upharpoonright_{\sigma}%
  \ifthenelse{\equal{#1}{}}{}{\mskip-5mu\left(#1\right)}
}
\newcommand*{\IDX}{\mathfrak{I}}
\newcommand*{\BCT}{\mathbb{T}}
\newcommand*{\VTX}{\mathcal{V}}
\newcommand*{\EDG}{\mathcal{E}}


%% file: content.tex
\section{Introduction}\label{sec:intro}

The philosophy of space-time methods is to consider space and time as components
of space-time rather than disconnected entities. Space-time finite elements are
based on meshes of the $d+1$-dimensional space-time domain, where $d\in\NN$
denotes the number of spatial dimensions. Especially over the course of the
last decade, space-time finite element methods have achieved remarkable progress
\cite{Neumueller2011,Steinbach2015,Gopalakrishnan2017a,Gopalakrishnan2019}.
Advantages of this methodology are the natural treatment of non-stationary
domains \cite{Wang2015,Langer2016}, adaptivity \cite{Doerfler2016,Poelz2019} and
efficient parallelization techniques \cite{Gander2016,Neumueller2013}.
In the context of hyperbolic problems, space-time approaches facilitate locally
explicit solution strategies exploiting causality and finite speed propagation
\cite{Gopalakrishnan2015,Perugia2020,Uengoer2002}.

While boundary integral equations (BIEs) have proven to be a compelling device
for exterior scattering problems or transparent boundary conditions
\cite{Abboud2011}, the development of genuine space-time boundary element
methods (BEMs) is in its infancy. Typical time domain BEMs are based on
semi-discretization. In particular, they employ trial functions which are the
product of separate functions in space and time
\cite{Costabel2017,Ha-Duong2003,Davies2004,Sauter2013,Gimperlein2018}. An
earlier attempt at relinquishing this product structure is due to Frangi
\cite{Frangi2000}, who exploits ``causal'' shape functions to discretize BIEs
of the wave equation for $d=2$. These functions can be interpreted as a
predecessor to trial functions defined on unstructured space-time meshes. By
giving up the usual product structure, however, one is confronted with more
complicated integrals. The evaluation of these integrals is a major obstacle,
stalling the practical development of space-time BEMs.

In the context of BIEs of parabolic problems, Tausch and collaborators
\cite{Manson2019,Tausch2019} are actively developing quadrature techniques for
these integrals. In \cite{Poelz2019a}, we proposed a tentative space-time BEM
for the wave equation for $d=3$. Integral formulations of the wave equation,
especially for odd $d\geq 3$, are of extraordinary structure, reverberating
through their name retarded potential boundary integral equations (RPBIEs). The
space-time methodology is particularly apt for treating the distinguished
nature of RPBIEs. Therefore, this paper is intended to advance our earlier work.

The novelty of this paper lies in the utilization of space-time boundary
elements to discretize variational formulations of RPBIEs. Although the
mathematical analysis of Galerkin methods for RPBIEs is yet incomplete
\cite{Joly2017}, they have already been applied successfully
\cite{Banz2016,Veit2016}. The integral operators acting on the surface density
$w:\Sigma\to\RR$ ($\Sigma$ is the space-time boundary) are of the form
\begin{equation*}
  \opTk{k} w:\TX{x}\mapsto\intOp{\mathcal{Q}(\TX{x})}{k(x,y)w(\TX{y})}%
  {S(\TX{y})}
  .
\end{equation*}
Here, $\TX{x},\TX{y}\in\RR[4]$ are points in space-time with spatial components
$x,y\in\RR[3]$ and $k:\RR[3]\times\RR[3]\to\RR$ is the integral kernel. The set
$\mathcal{Q}(\TX{x})$ is the intersection of $\Sigma$ and a quadratic
hypersurface, namely the backward light cone, which depends on $\TX{x}$.
Energetic bilinear forms with trial and test functions $w,v$ read
\cite{Aimi2009,Aimi2010}
\begin{equation*}
  (w,v)\mapsto
  \intOp{\Sigma}{\opTk{k} w(\TX{x})\partial_t v(\TX{x})}{S(\TX{x})}
  = \intOp{\Sigma}{\intOp{\mathcal{Q}(\TX{x})}{k(x,y)w(\TX{y})
      \partial_tv(\TX{y})}{S(\TX{y})}}{S(\TX{x})}
  .
\end{equation*}
From here on, we refer to these integrals as inner (integral operator
$\opTk{k}$) and outer (Galerkin testing). The perhaps most successful
quadrature techniques for BIEs of elliptic problems treat both integrals
together as one high-dimensional integral \cite{Erichsen1998,Sauter2011}.
While this approach has compelling advantages, the design of such
high-dimensional quadrature methods for hyperbolic problems is complicated due
to the nonlinear behavior of $\TX{x}\mapsto\mathcal{Q}(\TX{x})$. This is the
reason why typical quadrature schemes employed in classical
semi-discretizations of RPBIEs treat these integrals separately
\cite{Aimi2013,Gimperlein2018a}. The present paper stays in line with these
approaches in the sense that the inner and outer integral are treated
individually. On the one hand, an alternative to the quadrature scheme for the
inner integral we developed in \cite{Poelz2019a} is proposed. On the other
hand, a suitable formula for the outer integral and a tentative numerical
integration method are discussed.

The paper is organized as follows. In \cref{sec:rpbie}, we exhibit the model
initial-boundary value problem, two related RPBIEs, and their variational
formulations. \Cref{sec:disc} discusses space-time boundary
elements and quadrature techniques for RPBIEs. An algorithm which aims at the
efficient implementation of retarded potentials is presented in \cref{sec:algo}.
The purpose of \cref{sec:numex} is to verify the proposed schemes via numerical
experiments. \Cref{sec:conc} provides a brief conclusion of this work.

\section{Retarded potential boundary integral equations}%
\label{sec:rpbie}


Let $\RR[n],n\in\NN$ be equipped with the usual Euclidean inner product
$\eIp{\cdot}{\cdot}$ and induced norm $\norm{\cdot}$. The unit sphere is
denoted $\nSphere^{n-1}\coloneqq\{x\in\RR[n]:\norm{x}=1\}$ and we abbreviate
$\nSphere\coloneqq\nSphere^2$.
Consider a bounded open domain $\Omega^-\subset\RR[3]$ whose exterior is
denoted $\Omega^+\coloneqq\RR[3]\setminus\overline{\Omega^-}$. The Lipschitz
boundary $\Gamma\coloneqq\partial\Omega^-$ is equipped with the unit outward
normal vector field $\nu_\Gamma :\Gamma\to\nSphere$. Let $\intext\in\{+,-\}$ and
$\sDist[\Gamma]:\RR[3]\to\RR$ be the signed distance function of $\Gamma$
defined by $\sDist[\Gamma]:x\mapsto\intext \inf_{y\in\Gamma}\norm{x-y}$ for
$x\in\Omega^\intext$.
Throughout this work, time coordinates are defined as geometrized time, i.e,
the product of ordinary time and wave velocity, see
\cite[Section 2]{Poelz2019a}. Let $T>0$ be the simulation end
time and $Q^{\intext}\coloneqq (0,T)\times\Omega^{\intext}$ be the space-time
cylinder with lateral boundary $\Sigma\coloneqq (0,T)\times\Gamma$. To simplify
notation, we introduce the fixed decomposition of points in space-time
\begin{equation*}
  \RR[4]\ni\TX{x}\coloneqq (t,x) ,\quad \RR[4]\ni\TX{y}\coloneqq (\tau,y),
\end{equation*}
with times $t,\tau\in\RR$ and spatial components $x,y\in\RR[3]$. Since
$\Omega^\intext$ is stationary, $\sDist[\Sigma]:[0,T]\times\RR[3]\to\RR$ is the
time-invariant signed distance function of $\Sigma$ given by
$\sDist[\Sigma]:\TX{x}\mapsto\sDist[\Gamma](x)$. Moreover, the space-time
normal vector field $\nu_\Sigma:\Sigma\to\nSphere^3$ has vanishing time
component
\begin{equation*}
  \nu_\Sigma:\TX{x}\mapsto
  \begin{pmatrix}\nu_{\Sigma,t}(\TX{x})\\ \nu_{\Sigma,x}(\TX{x})\end{pmatrix}
  = \begin{pmatrix}0\\ \nu_\Gamma(x)\end{pmatrix}
  .
\end{equation*}
Let the (Lipschitz continuous) function $\phi_\Xi:\RR[4]\to\RR$ be defined by
\begin{equation}\label{eq:sgn_dist_cone}
  \phi_\Xi:\TX{x}\mapsto \norm{x}-t ,\quad
  \nabla \phi_\Xi:\TX{x}\mapsto \begin{pmatrix}
    \partial_t \phi_\Xi(\TX{x})\\ \nabla_x \phi_\Xi(\TX{x})\end{pmatrix}
  =  \begin{pmatrix} -1 \\ x/\norm{x}\end{pmatrix},
\end{equation}
where the gradient $\nabla$ is split into the time derivative $\partial_t$ and
the spatial gradient $\nabla_x$. The three-dimensional hypersurface 
$\Xi(\TX{x})\coloneqq\{\TX{y}\in\RR[4]:\phi_\Xi(\TX{x}-\TX{y})=0\}$ is the
backward light cone with apex at $\TX{x}$, see \cite[Fig. 1]{Poelz2019a}.

\subsection{Integral form of the wave equation}\label{sec:wave}

Let $\Box\coloneqq\partial_t^2-\Delta_x$ be the d'Alembertian and
$u:Q^\intext\to\RR$ be subject to the homogeneous wave equation
\begin{alignat}{4}
  \Box u &=0&\quad& \text{in}~ Q^\intext, \label{eq:pde}
  \\
  u = 0~\text{and}~\partial_t u &=0&\quad& \text{on}~ \{0\}\times\Omega^\intext.
  \label{eq:ic}
\end{alignat}
As a model problem consider Dirichlet boundary conditions with given datum
$g:\Sigma\to\RR$
\begin{equation}\label{eq:bc}
  \trDiri[\intext] u = g \quad \text{on}~\Sigma,
\end{equation}
where $\trDiri[\intext]$ denotes the trace operator, see \cite{McLean2000}.
The normal derivative of $u$ is denoted by $\trNeum[\intext] u$ and it holds
$\trNeum[\intext]:u\mapsto\eIp{\nu_\Gamma}{\trDiri[\intext]\nabla_x u}$ if $u$
is sufficiently smooth.
In this paper, we focus on boundary integral representations of the
solution of \cref{eq:pde,eq:ic,eq:bc}. The involved integral operators
employ the forward fundamental solution $\fSol$ of the d'Alembertian in three
spatial dimensions \cite[Operator 51]{Ortner1980a}
\begin{equation*}
  \fSol:\TX{x}\mapsto\frac{\delta_0(t-\norm{x})}{4\pi\norm{x}} =
  \frac{\delta_0\circ \phi_\Xi(\TX{x})}{4\pi\norm{x}},
\end{equation*}
where $\delta_0$ denotes the Dirac delta function. We tacitly exploit
$\delta_0(t)=\delta_0(-t)$ for any $t\in\RR$.

\begin{definition}\label{def:retpot}
  Let $\TX{x}\in Q^-\cup Q^+$. Define for sufficiently smooth $w:\Sigma\to\RR$
  the retarded single layer potential
  \begin{equation}\label{eq:slp}
    \potSl w(\TX{x})\coloneqq
    \intOp{\Sigma}{\fSol(\TX{x}-\TX{y}) w(\TX{y})}{S(\TX{y})} =
    \intOp{\Sigma}{k_1(x,y)w(\TX{y})\delta_0\circ \phi_\Xi(\TX{x-y})}{S(\TX{y})}
  \end{equation}
  and for sufficiently smooth $v:\Sigma\to\RR$ the retarded double layer
  potential
  \begin{equation}\label{eq:dlp}
    \potDl v(\TX{x}) \coloneqq
    \intOp{\Sigma}{\trNeum[][y] \fSol(\TX{x}-\TX{y})v(\TX{y})}{S(\TX{y})}
    =
    \intOp{\Sigma}{\left(k_3(x,y)v(\TX{y})+k_2(x,y)\partial_\tau v(\TX{y})%
      \right)\delta_0\circ \phi_\Xi(\TX{x-y})}{S(\TX{y})}
  \end{equation}
  with the kernel functions $k_i : \RR[3]\times\RR[3]\to\RR, i =1,2,3$
  \begin{equation}\label{eq:kfunc}
    k_1:(x,y)\mapsto\frac{1}{4\pi\norm{x-y}} ~,\quad
    k_2:(x,y)\mapsto \frac{\eIp{n_\Gamma(y)}{x-y}}{4\pi\norm{x-y}^2}~,\quad
    k_3:(x,y)\mapsto \frac{\eIp{n_\Gamma(y)}{x-y}}{4\pi\norm{x-y}^3}
    .
  \end{equation}
\end{definition}

In \cref{sec:rp_lc}, the formulas of \cref{def:retpot} are revisited, abolishing
the Dirac delta functions. It holds $\Box\potSl w=0$ and $\Box\potDl v=0$ in
$Q^-\cup Q^+$ for any (admissible) $w$ and $v$, respectively. Furthermore,
Kirchhoff's formula
$u=\intext\left(\potDl\trDiri[\intext]u -\potSl\trNeum[\intext]u\right)$
represents solutions of \cref{eq:pde,eq:ic} uniquely by their
Cauchy data $(\trDiri[\intext]u,\trNeum[\intext]u)$, see
\cite[Section 3.5]{Sayas2016}. Application of the trace induces the retarded
single layer and double layer boundary integral operators 
\begin{equation*}
  \trDiri[\intext] \potSl = \bioSl ,~\quad
  \trDiri[\intext] \potDl = \intext\tfrac{1}{2}\opId+\bioDl,
\end{equation*}
where the latter formula with the factor $1/2$ holds almost everywhere on
$\Sigma$. Note that the integral representations of $\bioSl$ and $\bioDl$ are
given by \cref{eq:slp,eq:dlp}, respectively (for $\TX{x}\in\Sigma$). In this
work, we examine two different approaches to solve \cref{eq:pde,eq:ic,eq:bc}
by means of BIEs. They are explained in \cref{tab:bie} to provide a concise
overview.

\begin{table}[ht]\centering
  \setlength\tabcolsep{18pt} 
  \ra{1.2} 
  \begin{tabular}{@{}ll@{\hskip 3mm}lll@{}}\toprule
    method & \multicolumn{2}{l}{unknown surface density} & BIE to be solved
    & solution of \cref{eq:pde,eq:ic,eq:bc}
    \\ \midrule
    indirect & $w$ & $\dots$ proxy density & $\bioSl w = g$ & $u = \potSl w$
    \\
    direct & $\trNeum[\intext] u$ & $\dots$ Neumann trace &
    $\bioSl \trNeum[\intext] u = -\intext\tfrac{1}{2}g + \bioDl g$ &
    $u = \intext\left(\potDl g -\potSl\trNeum[\intext]u\right)$
    \\ \bottomrule
  \end{tabular}
  \caption{Two approaches to solve \cref{eq:pde,eq:ic,eq:bc} via BIEs. The
    indirect method is based on the ansatz $u=\potSl w$, whose trace
    yields the BIE $\trDiri[\intext]\potSl w=g$. In contrast, the direct
    approach employs Kirchhoff's formula
    $u=\intext\left(\potDl g-\potSl\trNeum[\intext]u\right)$, describing the
    wave field in terms of its Cauchy data. The yet unknown Neumann trace
    $\trNeum[\intext] u$ is determined via the trace of Kirchhoff's formula.}
  \label{tab:bie}
\end{table}

In order to enable space-time Galerkin discretizations, suitable variational
formulations of the RPBIEs in \cref{tab:bie} are required. Finding compelling
space-time bilinear forms of RPBIEs has been the goal of multiple research
efforts \cite{Bamberger1986,Aimi2009,Joly2017}. However, computable bilinear
forms proven to be coercive in the same (Sobolev space) norm in which they are
bounded are elusive, see \cite[Theorems 3.1 and 3.3]{Aimi2009},
\cite[Corollary 4.6]{Joly2017}, and \cite[Theorem 3]{Costabel2017}. Therefore,
we resort to well-established bilinear forms which are of manageable complexity
(nevertheless nontrivial) and supported by experience. The chosen setting is as
in \cite[Theorem 3]{Costabel2017} with a weight of $\omega_I=0$, which
corresponds to \cite[Equations (29)--(31)]{Aimi2009}.
Let $\mathcal{H}_0$, $\mathcal{H}_1$ be appropriate Sobolev spaces and
$b_{\bioSl}:\mathcal{H}_0\times\mathcal{H}_0\to\RR$, 
$b_{\bioDl}:\mathcal{H}_1\times\mathcal{H}_0\to\RR$ be bilinear forms defined by
\begin{equation*}
  b_{\bioSl}: (w,v) \mapsto
  \intOp{\Sigma}{\bioSl w(\TX{x})\partial_t v(\TX{x})}{S(\TX{x})}
  ,\quad
  b_{\bioDl}: (w,v) \mapsto
  \intOp{\Sigma}{\bioDl w(\TX{x})\partial_t v(\TX{x})}{S(\TX{x})}
  .
\end{equation*}
The sole purpose of the abstract spaces $\mathcal{H}_0$ and $\mathcal{H}_1$ is
to distinguish qualitative properties of the input densities of $b_{\bioSl}$ and
$b_{\bioDl}$ in the discretization process of \cref{sec:disc}. For given
Dirichlet datum $g$ the functional $b_{\opId}(g,\cdot):\mathcal{H}_0\to\RR$ is
defined by $w\mapsto\intOp{\Sigma}{g \partial_t w\,}{S}$, where the subscript
$\opId$ denotes the identity map. A variational formulation of the indirect
approach in \cref{tab:bie} is:
\begin{equation}\label{eq:vf_ind}
  \text{Given}~g\in\mathcal{H}_1,~\text{find}~w\in\mathcal{H}_0 :~
  b_{\bioSl}(w,v) = b_{\opId}(g,v) \quad \forall v\in\mathcal{H}_0.
\end{equation}
For the direct method in \cref{tab:bie} we use the formulation:
\begin{equation}\label{eq:vf_dir}
  \text{Given}~g\in\mathcal{H}_1,
  ~\text{find}~\trNeum[\intext] u\in\mathcal{H}_0 :~
  b_{\bioSl}(\trNeum[\intext] u,v) =
  -\intext\tfrac{1}{2}b_{\opId}(g,v) + b_{\bioDl}(g,v)
  \quad\forall v\in\mathcal{H}_0
  .
\end{equation}
As shown in \cite[Propositions 3.4 and 3.7]{Joly2017}, a bilinear form similar
to $b_{\bioSl}$ is positive definite iff $T$ is sufficiently small, however, its
induced norm is not equivalent to the $\spaceHs[s][\Sigma]$-norm for any
$s\in\RR$ \cite[Theorem 3.1]{Aimi2009}. Still, promising numerical evidence is
reported in \cite{Ha-Duong2003,Aimi2012}.

\subsection{Retarded layer potential integrals from the light cone's
  perspective}
\label{sec:rp_lc}

In this segment, we recast the integral operators of \cref{def:retpot} to
a natural representation in the space-time context. In \cite{Poelz2019a}, we
employ local parametrizations of the space-time boundary $\Sigma$ to derive a
suitable formula for retarded potentials. In this work, we seek a
representation in terms of the light cone $\Xi$ instead.
\begin{theorem}[Coarea formula]\label{thm:coarea}
  Let $m,n\in\NN$ with $m>n$, $f:\RR[m]\to\RR[n]$ be Lipschitz continuous,
  $J_f(x)$ be its $n$-dimensional Jacobian at $x\in\RR[m]$, and
  $g:\RR[m]\to\RR$ be integrable. It holds
  \begin{equation*}
    \intOp{\RR[m]}{g(x)J_f(x)}{x} =
    \intOp{\RR[n]}{\intOp{f^{-1}\{y\}}{g(x)}{S(x)}}{y}
    .
  \end{equation*}
\end{theorem}
A proof of the coarea formula can be found in
\cite[Theorem 3.2.12]{Federer1996}. Assuming that $f$ and $g$ are as in
\cref{thm:coarea} with the addition that the function $\RR[m]\to\RR$,
$x\mapsto g(x)/J_f(x)$ is integrable, we may write
\begin{equation*}
  \intOp{\RR[m]}{g(x)\delta_0\circ f(x)}{x} =
  \intOp{\RR[n]}{\intOp{f^{-1}\{y\}}{\frac{g(x)\delta_0\circ f(x)}%
      {J_f(x)}}{S(x)}}{y} =
  \intOp{\RR[n]}{\delta_0(y)\intOp{f^{-1}\{y\}}{\frac{g(x)}{J_f(x)}%
    }{S(x)}}{y}
  .
\end{equation*}
This, in combination with the sifting property of $\delta_0$, leads to
\begin{equation}\label{eq:co_del}
  \intOp{\RR[m]}{g(x)\delta_0\circ f(x)}{x} =
  \intOp{f^{-1}\{0\}}{\frac{g(x)}{J_f(x)}}{S(x)},
\end{equation}
if the integral on the right hand side exists. A formula similar to
\cref{eq:co_del} for $n=1$ can be found in \cite[Theorem 6.1.5]{Hoermander2003}.
Let $f\coloneqq\sDist[\Sigma]$ and $g:\RR[4]\to\RR$ be such that
$\trDiri g\coloneqq\trDiri[-]g=\trDiri[+]g$ holds. In this case,
\cref{eq:co_del} yields
\begin{equation}\label{eq:co_sigma}
  \intOp{\RR[4]}{g(\TX{x})\delta_0\circ\sDist[\Sigma](\TX{x})}{\TX{x}} =
  \intOp{\sDist[\Sigma]^{-1}\{0\}}{
    \frac{\trDiri g(\TX{x})}{\norm{\trDiri\nabla\sDist[\Sigma](\TX{x})}}}%
  {S(\TX{x})} = \intOp{\Sigma}{\trDiri g(\TX{x})}{S(\TX{x})},
\end{equation}
where we used $\trDiri\nabla\sDist[\Sigma]=\nu_\Sigma$. We turn our attention
to the operators in \cref{def:retpot} and introduce an operator that unifies
the integral formulas of $\potSl$, $\potDl$, $\bioSl$, and $\bioDl$. Let
$\TX{x}\in\RR[4]$ and $k:\RR[3]\times\RR[3]\to\RR$ be a kernel function as in
\cref{eq:kfunc}. For sufficiently smooth $f:\RR[4]\to\RR$ with bounded support
we define the retarded Newtonian potential $\opNk{k}$ by
\begin{equation}\label{eq:tk_vol_del}
  \opNk{k} f(\TX{x}) \coloneqq
  \intOp{\RR[4]}{k(x,y)f(\TX{y}) \delta_0\circ\phi_\Xi(\TX{x}-\TX{y})}{\TX{y}}
  .
\end{equation}
Applying \cref{eq:co_del} and \cref{eq:sgn_dist_cone} to \cref{eq:tk_vol_del}
yields
\begin{equation}\label{eq:tk_vol}
  \opNk{k} f(\TX{x}) = \intOp{\TX{y}\in\RR[4]:\,\phi_\Xi(\TX{x}-\TX{y})=0}%
  {\frac{k(x,y)f(\TX{y})}{\norm{\nabla_{\TX{y}}\phi_\Xi(\TX{x}-\TX{y})}}}%
  {S(\TX{y})}=\frac{1}{\sqrt{2}}\intOp{\Xi(\TX{x})}{k(x,y)f(\TX{y})}{S(\TX{y})}
  .
\end{equation}
For sufficiently smooth $w:\Sigma\to\RR$ we define analogously the retarded
layer potential $\opTk{k}$ by
\begin{equation}\label{eq:tk_surf_del}
  \opTk{k} w(\TX{x}) \coloneqq
  \intOp{\Sigma}{k(x,y)w(\TX{y})\delta_0\circ\phi_\Xi(\TX{x}-\TX{y})}{S(\TX{y})}
  ,
\end{equation}
which models the operators in \cref{def:retpot} via $\potSl=\opTk{k_1}$ and
$\potDl=\opTk{k_3}+\opTk{k_2}\partial_t$.
For given $w:\Sigma\to\RR$ consider an extension $\widetilde{w}:\RR[4]\to\RR$
such that $w=\trDiri[-]\widetilde{w}=\trDiri[+]\widetilde{w}$ holds. Insertion
of $f\coloneqq \widetilde{w}\delta_0\circ\sDist[\Sigma]$ in \cref{eq:tk_vol_del}
in conjunction with \cref{eq:co_sigma} yields the identity
$\opNk{k}\left(\widetilde{w}\delta_0\circ\sDist[\Sigma]\right)=\opTk{k} w$,
where $\opTk{k} w$ is as in \cref{eq:tk_surf_del}. Application of
\cref{eq:tk_vol} leads to the desired representation
\begin{equation}\label{eq:tk_surf}
  \opTk{k} w(\TX{x}) =
  \opNk{k}\left(\widetilde{w}\delta_0\circ\sDist[\Sigma]\right)
  (\TX{x}) = \frac{1}{\sqrt{2}}\intOp{\Xi(\TX{x})}{k(x,y)\widetilde{w}(\TX{y})%
    \delta_0\circ\sDist[\Sigma](\TX{y})}{S(\TX{y})}.
\end{equation}
By recasting the potentials of \cref{def:retpot} to the form \cref{eq:tk_surf}
we have yet traded the Dirac delta on $\Xi(\TX{x})$ for a Dirac delta on
$\Sigma$. We incorporate a parametrization of $\Xi(\TX{x})$ to obtain a
computationally sensible formula. In the following, $SO(3)$ denotes the special
orthogonal matrix group in three spatial dimensions.

\begin{definition}\label{def:lc_para}
  Define the parameter domain $\mathcal{P}\coloneqq[0,\infty)\times[0,2\pi)
  \times[0,\pi]$ and parameters
  $\zeta\coloneqq(\rho,\varphi,\theta)\in\mathcal{P}$.
  For given $\TX{x}\in\RR[4]$, $r_0> 0$, and $R\in SO(3)$ define
  $\psi_{\TX{x}}:\mathcal{P}\to\RR[4]$ by
  \begin{equation*}
    \psi_{\TX{x}} : \zeta\mapsto \TX{x}-r_0\rho
    \begin{pmatrix} 1\\ R e_{\nSphere}(\varphi,\theta) \end{pmatrix}
    ,
  \end{equation*}
  where $e_{\nSphere}:[0,2\pi)\times[0,\pi]\to\nSphere$ is defined by
  $e_{\nSphere} : (\varphi,\theta) \mapsto
  \begin{pmatrix} \cos\varphi\sin\theta& \sin\varphi\sin\theta&
    \cos\theta\end{pmatrix}^\top$.
\end{definition}
Note that instead of using $e_{\nSphere}$ as stated in \cref{def:lc_para}, any
smooth parametrization of $\nSphere$ would suffice for our purposes, cf.
\cite[Definition 2.20]{Poelz2021}. For instance, the domain of the azimuthal
angle $\varphi$ could be defined as $[\alpha,\alpha+2\pi)$ for any
$\alpha\in\RR$.
The map $\psi_{\TX{x}}:\mathcal{P}\to\Xi(\TX{x})$ is surjective, its restriction
to the dense subset $(0,\infty)\times[0,2\pi)\times(0,\pi)\subset\mathcal{P}$
is injective, and its Jacobian reads
$J_{\psi_{\TX{x}}}:\zeta\mapsto\sqrt{2}r_0^3\rho^2\sin\theta$,
see \cite[Lemma 2.21, Proposition 3.15]{Poelz2021}. Convenient choices for
$r_0$ and $R$ in \cref{def:lc_para} are provided in \cref{def:choice_r}.
The parametrization $\psi_{\TX{x}}:\mathcal{P}\to\Xi(\TX{x})$ can be
used to transform the integral along $\Xi(\TX{x})$ in \cref{eq:tk_surf}
\begin{equation}\label{eq:int_para_lc}
  \intOp{\Xi(\TX{x})}{k(x,y)\widetilde{w}(\TX{y})
    \delta_0\circ\sDist[\Sigma](\TX{y})}{\TX{y}} =
  \intOp{\mathcal{P}}{k(x,\cdot)\circ\lvert_{y}\psi_{\TX{x}}(\zeta)
    \widetilde{w}\circ\psi_{\TX{x}}(\zeta)
    \delta_0\circ\sDist[\Sigma]\circ\psi_{\TX{x}}(\zeta)
    J_{\psi_{\TX{x}}}(\zeta)}{\zeta},
\end{equation}
where $\lvert_{y}\psi_{\TX{x}}(\zeta)$ is the spatial component of
$\TX{y}\coloneqq\psi_{\TX{x}}(\zeta)$. Application of \cref{eq:co_del} to
\cref{eq:int_para_lc} leads with \cref{eq:tk_surf} to
\begin{equation}\label{eq:tk_para}
  \opTk{k} w(\TX{x}) =
  \intOp{\psi_{\TX{x}}^{-1}(\Xi(\TX{x})\cap\Sigma)}%
  {k(x,\cdot)\circ\lvert_y\psi_{\TX{x}}(\zeta)%
    w\circ\psi_{\TX{x}}(\zeta)\frac{J_{\psi_{\TX{x}}}(\zeta)}%
    {\sqrt{2}\norm{\nabla\left(\sDist[\Sigma]\circ\psi_{\TX{x}}(\zeta)\right)}}}%
  {S(\zeta)},
\end{equation}
where $\psi_{\TX{x}}^{-1}(\Xi(\TX{x})\cap\Sigma) =
\{\zeta\in\mathcal{P}:\sDist[\Sigma]\circ\psi_{\TX{x}}(\zeta)=0\}$ is the subset
of $\Sigma$ lit by $\Xi(\TX{x})$ in parameter coordinates.
In other words, \cref{eq:tk_para} shows that retarded layer potentials
integrate along the intersection of $\Xi(\TX{x})$ and $\Sigma$. While
\cref{eq:tk_para} is based on a parametrization of $\Xi(\TX{x})$, we derive
in \cite[Equation (3.4b)]{Poelz2019a} an alternative representation of
$\opTk{k}$ based on piecewise parametrizations of $\Sigma$. Both integral
representations of $\opTk{k}$, \cref{eq:tk_para} and
\cite[Equation (3.4b)]{Poelz2019a}, are valid for $C^1_{\text{pw}}$-hypersurfaces
$\Sigma$ in the sense of \cite[Definition 2.2.10]{Sauter2011}. The expression
for $\opTk{k}$ in \cref{eq:tk_para} is specialized to piecewise flat boundary
decompositions in \cref{sec:q_inner}.

\section{Space-time discretization and numerical evaluation of
  retarded layer potentials}\label{sec:disc}

As already indicated, the novelty of the proposed method lies in the
utilization of space-time boundary element spaces as in \cite{Poelz2019a}. The
space-time boundary $\Sigma$ is represented by
$\Sigma_N\coloneqq\{\sigma_i\}_{i=1}^N$, a mesh composed of $N\in\NN$ open
nonoverlapping tetrahedrons $\sigma$. We refer to the subsets
$\sigma\subset\Sigma$ as panels (not elements), see
\cite[Section 1.2]{Sauter2011} and \cite[Section 2.3]{Ciarlet2002}.
The mesh size is denoted $h\coloneqq\max_{\sigma\in\Sigma_N}\diam\sigma$.
Simplex space-time meshes are constructed via the algorithm outlined in
\cite{Karabelas2015} and we resort to lowest order trial spaces.
\begin{definition}\label{def:st_be}
  Let $\mathbb{P}_n(\sigma)$ be the space of polynomials of order up to
  $n\in\NN_0\coloneqq\NN\cup\{0\}$ in the tetrahedron $\sigma$. Define the
  (discontinuous) space of indicator functions and the space of (continuous)
  hat functions by
  \begin{equation*}
    S_h^0(\Sigma_N)\coloneqq\prod\nolimits_{\sigma\in\Sigma_N}
    \mathbb{P}_0(\sigma) ,\quad
    S_h^1(\Sigma_N)\coloneqq\prod\nolimits_{\sigma\in\Sigma_N}
    \mathbb{P}_1(\sigma)\cap C(\Sigma)
    .
  \end{equation*}
  The subspace of $S_h^1(\Sigma_N)$ with homogeneous initial conditions is
  defined by $V_h^1(\Sigma_N)\coloneqq
  \{v\in S_h^1(\Sigma_N):v\rvert_{\{0\}\times\Gamma}=0\}$. It holds
  $\dim S_h^0(\Sigma_N) = N$ and $\dim S_h^1(\Sigma_N)$ equals the number of
  vertices in $\Sigma_N$.
\end{definition}
The spaces of \cref{def:st_be} are labeled space-time boundary element spaces
because there is no inherent distinction between space and time variables. The
space $S_h^0(\Sigma_N)$ is intended for discretization of $\mathcal{H}_0$, while
$V_h^1(\Sigma_N)$ is used to approximate functions in $\mathcal{H}_1$.
Consequently, the discretized version of \cref{eq:vf_ind} reads:
\begin{equation}\label{eq:df_ind}
  \text{Given}~g\in\mathcal{H}_1,~\text{find}~w_h\in S_h^0(\Sigma_N) :~
  b_{\bioSl}(w_h,v_h) =  b_{\opId}(g,v_h) \quad \forall v_h\in S_h^0(\Sigma_N)
  .
\end{equation}
In \cref{eq:vf_dir} the integral operator $\bioDl$ acts on the given Dirichlet
data $g$. In such cases it is common practice in BEMs for elliptic problems to
approximate the data, see, e.g., \cite[Chapter 12]{Steinbach2008}. To this end,
we employ the $\LTSig$-orthogonal projection
$\opQ_h^1 :\LTSig\to V_h^1(\Sigma_N)$ with the usual inner product
$\gIp{\cdot}{\cdot}{\LTSig}: (w,v)\mapsto\intOp{\Sigma}{w v\,}{S}$. The
projection $\opQ_h^1 g\in V_h^1(\Sigma_N)$ of $g\in\LTSig$ is the unique
solution of
\begin{equation*}
  \gIp{\opQ_h^1 g}{v_h}{\LTSig} = \gIp{g}{v_h}{\LTSig}
  \quad \forall v_h\in V_h^1(\Sigma_N)
  .
\end{equation*}
Assuming $g\in\LTSig$ holds, the discretization of \cref{eq:vf_dir} with
$w_h\approx\trNeum[\intext] u$ reads:
\begin{equation}\label{eq:df_dir}
  \text{Given}~g\in\LTSig,~\text{find}~w_h\in S_h^0(\Sigma_N) :~
  b_{\bioSl}(w_h,v_h) =
  -\intext\tfrac{1}{2} b_{\opId}(\opQ_h^1 g,v_h) + b_{\bioDl}(\opQ_h^1 g,v_h)
  \quad \forall v_h\in S_h^0(\Sigma_N)
  .
\end{equation}
The following sections are concerned with the numerical evaluation of the
involved operators.

\subsection{A quadrature method for the ``inner integral''}%
\label{sec:q_inner}

In this section, we devise a numerical integration scheme for
\cref{eq:tk_para} tailored to tetrahedral panels.

\begin{definition}[Inner integral \cref{eq:tk_para}]\label{def:int_inner}
  Let $\TX{x}\in\RR[4]$ be arbitrary but fixed and $\sigma$ be a tetrahedron
  embedded in $\RR[4]$ with normal vector $\nu\in\nSphere^3$. The unit outward
  conormal vectors of the four triangular faces of $\sigma$ are denoted
  $\nu_i\in\nSphere^3,i=1,\dots,4$ and satisfy $\eIp{\nu}{\nu_i}=0$. Let
  $k:\RR[3]\times\RR[3]\to\RR$ be as in \cref{eq:kfunc}, $w:\sigma\to\RR$ be
  analytic, and $\psi_{\TX{x}}:\mathcal{P}\to\Xi(\TX{x})$ be as in
  \cref{def:lc_para}. Define the integral kernel $k_\psi:\mathcal{P}\to\RR$ and
  the integral $\mathcal{I}$ by
  \begin{equation*}
    k_\psi:\zeta\mapsto\frac{k(x,\cdot)\circ\lvert_y\psi_{\TX{x}}(\zeta)%
      J_{\psi_{\TX{x}}}(\zeta)}%
    {\sqrt{2}\norm{\nabla \left(\sDist[\Sigma]\circ\psi_{\TX{x}}(\zeta)\right)}}
    ,\quad
    \mathcal{I} \coloneqq \intOp{\psi_{\TX{x}}^{-1}(\Xi(\TX{x})\cap\sigma)}%
    {k_\psi(\zeta) w\circ\psi_{\TX{x}}(\zeta)}{S(\zeta)}.
  \end{equation*}
\end{definition}

The notation introduced in \cref{def:int_inner} is employed throughout the
remainder of this section.
We denote the tangent hyperplane of the panel $\mathcal{T}_\sigma\coloneqq
\{\TX{y}\in\RR[4]:\eIp{\TX{y}-\TX{x}_\sigma}{\nu}=0\}$ for some
$\TX{x}_\sigma\in\overline{\sigma}$. Each triangular face of $\sigma$
induces a half-space $\{\TX{y}\in\RR[4]:\eIp{\TX{y}-\TX{x}_i}{\nu_i}<0\}$,
where $\TX{x}_i,i=1,\dots,4$ is a vertex in that face. The panel is the
intersection of these half-spaces and $\mathcal{T}_\sigma$
\begin{equation}\label{eq:pnl_ilf}
  \sigma = \left\{\TX{y}\in\mathcal{T}_\sigma :
    \eIp{\TX{y}-\TX{x}_i}{\nu_i} < 0 ~~\forall i=1,\dots,4
  \right\}
  .
\end{equation}
The condition $\psi_{\TX{x}}(\zeta)\in\mathcal{T}_\sigma$ with $\psi_{\TX{x}}$
as in \cref{def:lc_para} is equivalent to
\begin{align}
  0 &=\eIp{\TX{x}-r_0\rho \begin{pmatrix} 1\\
      R e_{\nSphere}(\varphi,\theta) \end{pmatrix}-\TX{x}_\sigma}{\nu}
  = \eIp{\TX{x}-\TX{x}_\sigma}{\nu}-r_0\rho
  \eIp{\begin{pmatrix} 1\\ e_{\nSphere}(\varphi,\theta)\end{pmatrix}}%
  {\begin{pmatrix} \nu_t\\ R^\top \nu_x\end{pmatrix}}
  \nonumber \\ &= \eIp{\TX{x}-\TX{x}_\sigma}{\nu}
  -r_0\rho \eIp{e_{\nSphere}(\varphi,\theta)}{R^\top \nu_x},
  \label{eq:psi_xi_in_hypl}
\end{align}
where the spatial and time components of $\nu$ are denoted by $\nu_x\in\RR[3]$
and $\nu_t=0$, respectively.
\begin{definition}\label{def:choice_r}
  For given normal vector $\nu_x\in\nSphere$ let $R\in SO(3)$ be such that
  $R^\top\nu_x=\begin{pmatrix}0&0&1\end{pmatrix}^\top$ holds. Furthermore,
  for given apex $\TX{x}\in\RR[4]$ and panel $\sigma\in\Sigma_N$ let
  $r_0\coloneqq \sup_{\TX{y}\in\sigma}(t-\tau)$ be the largest time separation
  between $\TX{x}$ and $\sigma$.
\end{definition}
Note that $\sup_{\TX{y}\in\sigma}(t-\tau)\leq 0$ implies
$\Xi(\TX{x})\cap\sigma\in\left\{\varnothing,\{\TX{x}\}\right\}$ and, therefore,
$\mathcal{I}=0$. As a consequence, the potential conflict between the
assumption $r_0>0$ in \cref{def:lc_para} and $r_0$ as in \cref{def:choice_r} is
of no practical significance. An explicit formula for $R$ in \cref{def:choice_r}
is provided in \cite[Remark 3.12]{Poelz2021}. Inserting
$e_{\nSphere}$ from \cref{def:lc_para} and $R$ from \cref{def:choice_r} into
\cref{eq:psi_xi_in_hypl} leads to
\begin{equation}\label{eq:para_hypl}
  \psi_{\TX{x}}(\zeta)\in\mathcal{T}_\sigma \quad \Leftrightarrow \quad
  \rho_0-\rho\cos\theta=0,
\end{equation}
with $\rho_0\coloneqq \eIp{\TX{x}-\TX{x}_\sigma}{\nu}/r_0$.
The partial derivatives of the level set function in \cref{eq:para_hypl}
$\phi:(\rho,\theta)\mapsto \rho_0-\rho\cos\theta$ are
$\partial_\rho\phi:(\rho,\theta)\mapsto-\cos\theta$ and
$\partial_\theta\phi:(\rho,\theta)\mapsto\rho\sin\theta$. We
restrict these derivatives to the solution of \cref{eq:para_hypl}, namely
$\rho_0 = \rho\cos\theta$, and obtain the maps $\rho\mapsto-\rho_0/\rho$ as
well as $\rho\mapsto\sqrt{\rho^2-\rho_0^2}$. While the magnitude of the first
derivative is monotonically decreasing, the latter is increasing. This shows
that for sufficiently small $\rho$ the $\rho$-direction is suitable for
parametrizing the solution of \cref{eq:para_hypl}, while for large $\rho$ the
$\theta$-direction becomes the better choice, see \cref{fig:rho_the_r0}. The
point where the partial derivatives are of equal magnitude is given by
\begin{equation*}
  %
  \rho_{\operatorname{eq}}=\frac{1}{\sqrt{2}} \left(
    \rho_0^2+\left(\rho_0^4+4\rho_0^2\right)^{1/2}\right)^{1/2}
  ,\quad
  \theta_{\operatorname{eq}}=\arccos\left(\rho_0/\rho_{\operatorname{eq}}\right)
  .
\end{equation*}
Define the two domains
\begin{equation*}
  \mathcal{D}_1\coloneqq \begin{cases}
    \left[0,\theta_{\operatorname{eq}}\right) \times [0,2\pi)
    &\text{if}~\rho_0 > 0,
    \\
    \left(\theta_{\operatorname{eq}},\pi\right] \times [0,2\pi)
    &\text{if}~\rho_0 < 0,
    \\
    \varnothing &\text{if}~\rho_0 = 0,
  \end{cases}
  \quad \mathcal{D}_2\coloneqq
  \left[\rho_{\operatorname{eq}},\infty\right)\times [0,2\pi),
\end{equation*}
and the parametrizations $\ell_i:\mathcal{D}_i\to\mathcal{P},i=1,2$ by
\begin{equation*}
  \ell_1:\begin{pmatrix}\theta\\ \varphi\end{pmatrix}\mapsto
  \begin{pmatrix}\rho_0/\cos\theta\\ \varphi\\ \theta\end{pmatrix}
  ,\quad \ell_2:\begin{pmatrix}\rho\\ \varphi\end{pmatrix}\mapsto
  \begin{pmatrix}\rho\\ \varphi\\ \arccos\left(\rho_0/\rho\right)\end{pmatrix}.
\end{equation*}
We have $\ell_1(\mathcal{D}_1)\cap\ell_2(\mathcal{D}_2)=\varnothing$ and
$\ell_1(\mathcal{D}_1)\cup\ell_2(\mathcal{D}_2)=
\psi_{\TX{x}}^{-1}(\Xi(\TX{x})\cap\mathcal{T}_\sigma)$ due to the careful
construction of these parametrizations, see \cref{fig:rho_the_li}. From
\cref{eq:pnl_ilf} we conclude $\psi_{\TX{x}}^{-1}(\Xi(\TX{x})\cap\sigma)=
\{\zeta\in\ell_1(\mathcal{D}_1)\cup\ell_2(\mathcal{D}_2):
\phi_\sigma\circ\psi_{\TX{x}}(\zeta)<0\}$, where $\phi_\sigma:\RR[4]\to\RR$ is
defined by
\begin{equation}\label{eq:dist_faces}
  \phi_\sigma: \TX{y}\mapsto \max_{i\in\{1,\dots,4\}}
  \eIp{\TX{y}-\TX{x}_i}{\nu_i}
  .
\end{equation}
This facilitates a parametrization of the integral in \cref{def:int_inner} via
$\ell_1$ and $\ell_2$
\begin{equation}\label{eq:int_inner_param}
  \mathcal{I}=\sum\nolimits_{i\in\{1,2\}}
  \intOp{\eta\in\mathcal{D}_i:\phi_\sigma\circ\psi_{\TX{x}}\circ\ell_i(\eta)<0}%
  {k_\psi\circ\ell_i(\eta)
    w\circ\psi_{\TX{x}}\circ\ell_i(\eta)J_{\ell_i}(\eta)}{\eta}
  ,
\end{equation}
which involves integrals in the implicitly defined subsets
$\mathcal{B}_i\coloneqq\{\eta\in\mathcal{D}_i:
\phi_\sigma\circ\psi_{\TX{x}}\circ\ell_i(\eta)<0\}$ of the rectangular
patches $\mathcal{D}_i$ for $i=1,2$. \Cref{def:choice_r} enables the
use of a finite patch $\mathcal{D}_2$ in \cref{eq:int_inner_param}: the
following choice is sufficient to capture the entire panel $\sigma$
\begin{equation*}
  \mathcal{D}_2\coloneqq\begin{cases}
    \left[\rho_{\operatorname{eq}},1\right]\times [0,2\pi)&
    \text{if}~\rho_{\operatorname{eq}}<1,\\
    \varnothing& \text{if}~\rho_{\operatorname{eq}}\geq 1.\end{cases}
\end{equation*}

\begin{figure}[htbp]
  \centering
  \subcaptionbox{solutions of \cref{eq:para_hypl} for various values of
    $\rho_0$\label{fig:rho_the_r0}}%
  {\includegraphics[height=55mm]{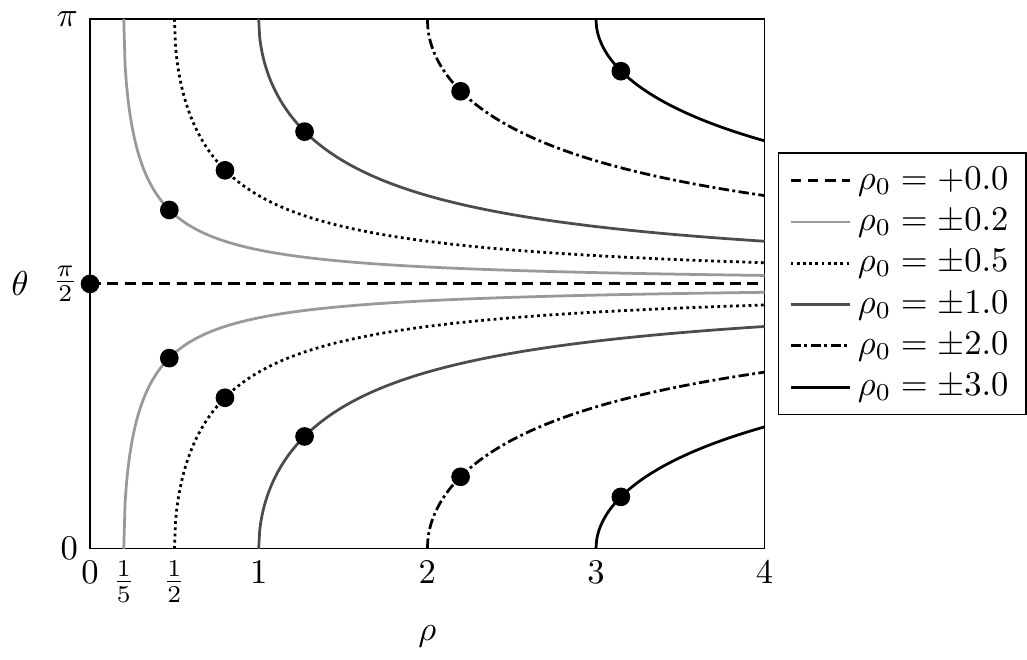}}
  \hfill
  \subcaptionbox{parametrizations $\ell_1$ and $\ell_2$ for $\rho_0=1/5$%
    \label{fig:rho_the_li}}%
  {\includegraphics[height=55mm]{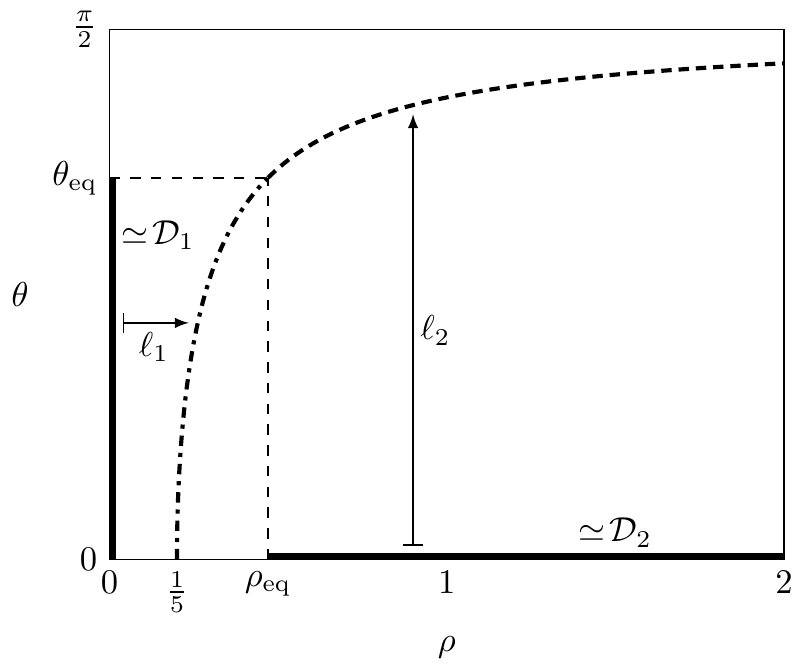}}
  \caption{In \cref{fig:rho_the_r0} curves beneath the line $\theta=\pi/2$
    are solutions of \cref{eq:para_hypl} for $\rho_0>0$, while those above
    correspond to $\rho_0<0$. The black dots indicate the points
    $(\rho_{\operatorname{eq}},\theta_{\operatorname{eq}})$,
    where the tangent to the curve has unit slope. \Cref{fig:rho_the_li}
    illustrates the piecewise parametrization of
    $\psi_{\TX{x}}^{-1}(\Xi(\TX{x})\cap\mathcal{T}_\sigma)$
    in terms of $\ell_1$ and $\ell_2$ for $\rho_0=1/5$. Depiction of the
    $\varphi$-component is omitted because the curve is merely extruded into
    the third direction.}
  \label{fig:rho_the}
\end{figure}

\begin{lemma}\label{lem:kernel_para}
  Let $k_\psi$ be as in \cref{def:int_inner} with kernel function
  $k_i,i=1,2,3$ as in \cref{eq:kfunc} and $R$ be as in \cref{def:choice_r}.
  The integral kernel $k_\psi\circ\ell_j,j=1,2$ is smooth in $\mathcal{D}_j$.
\end{lemma}
\begin{proof}
  Define $\mathcal{P}_\sigma^{\TX{x}}\coloneqq
  \psi_{\TX{x}}^{-1}(\Xi(\TX{x})\cap\mathcal{T}_\sigma)$ and let
  $k_i^\ast:\mathcal{P}_\sigma^{\TX{x}}\to\RR$ be given by $k_i^\ast:\zeta\mapsto
  k_i(x,\cdot)\circ\lvert_y\psi_{\TX{x}}(\zeta)$ with $k_i,i=1,2,3$ as in
  \cref{eq:kfunc} apart from the factor $4\pi$. For $R$ as in
  \cref{def:choice_r} and any $\zeta\in\mathcal{P}_\sigma^{\TX{x}}$ it holds
  \begin{equation*}
    \eIp{\nu_x}{x-\lvert_y\psi_{\TX{x}}(\zeta)}=
    r_0\rho\eIp{R^\top\nu_x}{e_{\nSphere}(\varphi,\theta)}=
    r_0\rho\cos\theta = r_0\rho_0
  \end{equation*}
  because $\zeta$ satisfies \cref{eq:para_hypl}. Since $k_i^\ast$
  is restricted to arguments $\zeta\in\mathcal{P}_\sigma^{\TX{x}}$, it follows
  $k_1^\ast:\zeta\mapsto (r_0\rho)^{-1}$ and
  $k_i^\ast:\zeta\mapsto\ r_0\rho_0 (r_0\rho)^{-i}$ for $i=2,3$.
  Define the function $k_{\psi,i}:\mathcal{P}_\sigma^{\TX{x}}\to\RR$ by
  \begin{equation*}
    k_{\psi,i}:\zeta\mapsto
    \frac{k_i^\ast(\zeta)J_{\psi_{\TX{x}}}(\zeta)}%
    {\sqrt{2}\norm{\nabla \left(\sDist[\Sigma]\circ\psi_{\TX{x}}(\zeta)\right)}}
  \end{equation*}
  which represents $k_\psi$ in \cref{def:int_inner}. The function
  $\sDist[\Sigma]$ corresponds to
  $\sDist[\mathcal{T}_\sigma]:\TX{x}\mapsto \eIp{\TX{x}-\TX{x}_\sigma}{\nu}$,
  where $\nabla\sDist[\mathcal{T}_\sigma]:\TX{x}\mapsto\nu$ is constant-valued.
  The matrix representation of $\nabla\psi_{\TX{x}}$ reads
  \begin{equation*}
    \nabla\psi_{\TX{x}}:\zeta\mapsto D_\psi(\zeta)\coloneqq
    -r_0 \begin{pmatrix} 1 & 0 & 0 \\ R e_{\nSphere}(\varphi,\theta) &
      \rho R \partial_\varphi e_{\nSphere}(\varphi,\theta) &
      \rho R \partial_\theta e_{\nSphere}(\varphi,\theta)
    \end{pmatrix} \in\RR[4\times 3]
  \end{equation*}
  and by the chain rule we have $\nabla\left(\sDist[\Sigma]\circ\psi_{\TX{x}}
  \right):\zeta\mapsto D_\psi^\top(\zeta)\nu$. We observe
  \begin{equation*}
    \nabla\left(\sDist[\Sigma]\circ\psi_{\TX{x}}\right):\zeta\mapsto
    -r_0\begin{pmatrix}
      \nu_t +\eIp{e_{\nSphere}(\varphi,\theta)}{R^\top \nu_x}\\
      \rho\eIp{\partial_\varphi e_{\nSphere}(\varphi,\theta)}{R^\top \nu_x}\\
      \rho\eIp{\partial_\theta e_{\nSphere}(\varphi,\theta)}{R^\top \nu_x}
    \end{pmatrix}
    .
  \end{equation*}
  For $R$ as in \cref{def:choice_r} we get
  $\norm{\nabla \left(\sDist[\Sigma]\circ\psi_{\TX{x}}(\zeta)\right)} = r_0
  \sqrt{(\cos\theta)^2+(\rho\sin\theta)^2}$ and
  \begin{equation}\label{eq:kernel_para_full}
    k_{\psi,1}:\zeta\mapsto
    \frac{r_0\rho\sin\theta}{\sqrt{(\cos\theta)^2+(\rho\sin\theta)^2}} ,\quad
    k_{\psi,i}:\zeta\mapsto \rho_0 r_0^{3-i} \rho^{2-i}
    \frac{\sin\theta}{\sqrt{(\cos\theta)^2+(\rho\sin\theta)^2}}, i=2,3,
  \end{equation}
  where we used $J_{\psi_{\TX{x}}}(\zeta) =\sqrt{2}r_0^3\rho^2\sin\theta$.
  If $\rho_0=0$ holds, it follows $k_{\psi,i}(\zeta)=0$ for $i=2,3$ and
  \cref{eq:para_hypl} implies $\theta=\pi/2$ (for $\rho\neq 0$), leading to
  $k_{\psi,1}(\zeta)=r_0$. For $\rho_0\neq 0$ the singularity of
  $k_i^\ast,i=1,2,3$ plays a role only for $\ell_1$ because $\ell_2$ maps to
  $\rho\geq\rho_{\operatorname{eq}}> 0$. For $\rho_0\neq 0$ we insert
  $\rho=\rho_0/\cos\theta$ in \cref{eq:kernel_para_full}, yielding
  \begin{equation*}
    k_{\psi,1}\circ\ell_1:(\theta,\varphi)\mapsto \rho_0r_0
    \frac{\tan\theta}{\sqrt{(\cos\theta)^2+(\rho_0\tan\theta)^2}}
    ,\quad
    k_{\psi,i}\circ\ell_1:(\theta,\varphi)\mapsto \rho_0^{3-i}r_0^{3-i}
    \frac{(\cos\theta)^{i-1}\tan\theta}%
    {\sqrt{(\cos\theta)^2+(\rho_0\tan\theta)^2}}
  \end{equation*}
  for $i=2,3$, confirming their smoothness for $\theta\neq\pi/2$
  ($\theta_{\operatorname{eq}}$ is bounded away from $\pi/2$ for $\rho_0\neq 0$).
\end{proof}
There exist several procedures for evaluating integrals like
\cref{eq:int_inner_param} accurately, see, e.g.,
\cite{Mueller2013,Saye2015,Fries2017}. The algorithm employed in this paper is
a combination of quadtree subdivision and exact parametrizations of the zero
level set. In a nutshell, it attempts to identify the shape of the subset of
$\mathcal{B}_i$ that lies in a quadtree cell among a few predefined scenarios.
For these admissible cases, exact parametrizations of the relevant subset are
constructed and the transformed integrals are approximated accurately by
standard tensor-Gauss quadrature rules. If this case identification fails, the
algorithm resorts to subdivision. Our approach is based on the method proposed
in \cite{Fries2016}, however, in contrast to the cited source, the zero level
set is parametrized by means of the ideal transformation discussed in
\cite{Lehrenfeld2016}. A similar approach is elaborated in \cite{Gfrerer2018}.
We denote the depth of the quadtree by $r_{\max}\in\NN_0$ and $n_G\in\NN$ is the
number of Gaussian quadrature points per direction. For each admissible
quadtree cell at most $2n_G^2$ quadrature points are employed. The reader is
referred to \cite[Section 3.8.1]{Poelz2021} for details regarding the
implementation.

\subsection{A quadrature method for the ``outer integral''}%
\label{sec:q_outer}

In order to evaluate the bilinear forms in \cref{eq:df_ind,eq:df_dir}, integrals
of the form $\intOp{\Sigma}{w \partial_t v_h}{S}$ have to be computed,
where $v_h\in S_h^0(\Sigma_N)$ and $w$ is either in $S_h^1(\Sigma_N)$ or defined
through the action of $\opTk{k}$. To this end, we consider a fixed
panel $\sigma\subset\Sigma$ with Lipschitz boundary $\partial\sigma$. The unit
outward conormal vector field $\nu_{\partial\sigma}:\partial\sigma\to\nSphere^3$
satisfies $\eIp{\nu_{\partial\sigma}(\TX{x})}{\nu_{\Sigma}(\TX{x})}=0$ for any
$\TX{x}\in\partial\sigma$ for which it exists. Let $v\in C^1(\Sigma)$ and
define
\begin{equation}\label{eq:ext_pnl_z}
  v\rvertPnl:\TX{x}\mapsto \begin{cases}v(\TX{x})&\text{if}~\TX{x}\in\sigma,\\
    0 &\text{otherwise},\end{cases}
\end{equation}
which jumps only across $\partial\sigma$. Note that $v\rvertPnl$ represents a
basis function of $S_h^0(\Sigma_N)$ if $v$ is constant-valued. For
$w\in C(\Sigma)$ we obtain from \cite[Theorem 3.1.9]{Hoermander2003}
\begin{equation*}
  \intOp{\Sigma}{w(\TX{x}) \partial_t v\rvertPnl[\TX{x}]}{S(\TX{x})} =
  \intOp{\sigma}{w(\TX{x})\partial_t v(\TX{x})}{S(\TX{x})}
  -\intOp{\partial\sigma}{\nu_{\partial\sigma,t}(\TX{x}) w(\TX{x}) v(\TX{x})}%
  {S(\TX{x})}
  ,
\end{equation*}
where $\nu_{\partial\sigma,t}$ is the time component of the unit outward
conormal vector. The application of the cited theorem is justified because time
is a tangential coordinate on $\Sigma$. For $w\in C(\Sigma)$
and $v\in\prod_{\sigma\in\Sigma_N}C^1(\sigma)$ we deduce
\begin{equation}\label{eq:int_prod_panel}
  \intOp{\Sigma}{w(\TX{x}) \partial_t v(\TX{x})}{S(\TX{x})} =
  \sum_{\sigma\in\Sigma_N}
  \intOp{\sigma}{w(\TX{x})\partial_t v\rvert_\sigma(\TX{x})}{S(\TX{x})}
  - \sum_{\sigma\in\Sigma_N} \intOp{\partial\sigma}%
  {\nu_{\partial\sigma,t}(\TX{x}) w(\TX{x}) v\rvert_{\partial\sigma}(\TX{x})}%
  {S(\TX{x})}
  ,
\end{equation}
where $v\rvert_{\partial\sigma}$ denotes the trace of the restriction
$v\rvert_\sigma$ to $\partial\sigma$. Let the bilinear form $b_{\opA}$,
$\opA\in\{\opId,\opTk{k}\}$ be defined by $b_{\opA} : (w,v)\mapsto
\intOp{\Sigma}{\opA w(\TX{x})\partial_t v(\TX{x})}{S(\TX{x})}$.
For $v_h\in S_h^0(\Sigma_N)$, i.e., $v_h\rvert_\sigma$ is constant-valued,
$b_{\opA}$ can be evaluated via \cref{eq:int_prod_panel}, leading to
\begin{equation}\label{eq:int_bform}
  b_{\opA} : (w,v_h) \mapsto -\sum_{\sigma\in\Sigma_N} \intOp{\partial\sigma}%
  {\nu_{\partial\sigma,t}(\TX{x})v_h\rvert_{\partial\sigma}(\TX{x})\opA w(\TX{x})}%
  {S(\TX{x})}
\end{equation}
if $\opA w$ is continuous across $\partial\sigma$. Note that the use of the
lowest order test space causes integrals on $\sigma$ to vanish in
\cref{eq:int_bform}. Since $\partial\sigma$ is composed of four $2$-simplices,
we only require a quadrature technique for triangles.
In classical time domain discretization schemes, certain singularities of
functions induced by retarded layer potentials have been studied, leading to
carefully developed quadrature schemes \cite{Stephan2008,Ostermann2010}.
Such an analysis in the space-time context could unveil the regularity of
the function $\TX{x}\mapsto \opTk{k}w(\TX{x})$ for different kernels $k$ and
regularity classes of $w$. These investigations might drive the design of
tailored quadrature schemes for \cref{eq:int_bform}. While such comprehensive
surveys lie beyond the scope of this work, the occurrence of singularities in
the function $\TX{x}\mapsto \opTk{k}w(\TX{x})$ is hinted in \cref{app:appendix}
by virtue of an example.
Due to the lack of smoothness, we suggest to apply composite midpoint rules in
order to evaluate \cref{eq:int_bform} for $\opA=\opTk{k}$, see
\cref{fig:int_tri_mid}.

\begin{figure}[htbp]
  \centering
  \includegraphics[width=.99\textwidth]{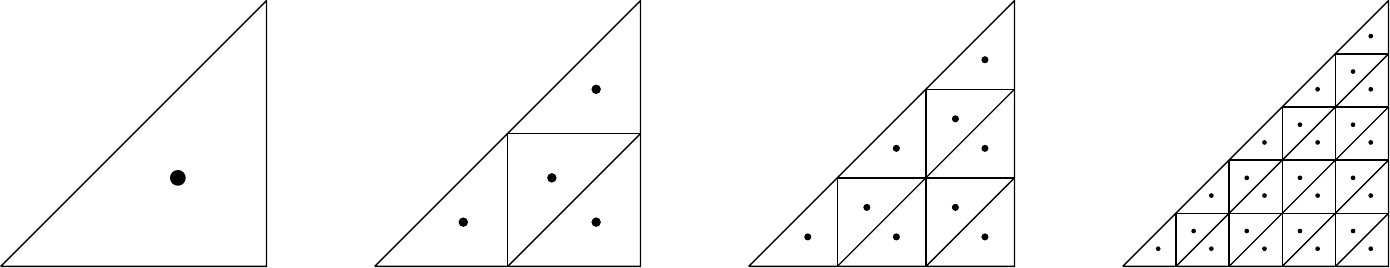}
  \caption{Illustration of quadrature rules employed for \cref{eq:int_bform};
    for $m_Q\in\NN$ the reference triangle is subdivided into $m_Q^2$ congruent
    triangles. Due to the lack of smoothness of the integrand, the midpoint
    rule is applied in each triangle. The triangles and midpoints, indicated by
    black dots, are displayed for $m_Q=1,2,3,5$ (from left to right).}
  \label{fig:int_tri_mid}
\end{figure}

\subsection{An algorithm for computing the set of lit panels efficiently}%
\label{sec:algo}

As discussed in \cref{sec:rp_lc}, retarded layer potentials evaluated at
$\TX{x}\in\RR[4]$ integrate along $\Xi(\TX{x})\cap\Sigma$. Given a mesh
$\Sigma_N$, the set of panels lit by the backward light cone is denoted
\begin{equation*}
  \Sigma_N^\Xi(\TX{x}) \coloneqq
  \{\sigma\in\Sigma_N:\Xi(\TX{x})\cap\sigma\neq\varnothing\}
  .
\end{equation*}
Both $\Sigma$ and $\Xi(\TX{x})$ are three-dimensional hypersurfaces, implying
that $\Xi(\TX{x})\cap\Sigma$ is two-dimensional, unless it degenerates. The
dimensions of these sets suggest that although $|\Sigma_N|=N$ holds, we expect
$|\Sigma_N^\Xi(\TX{x})|=\mathcal{O}(N^{2/3})$ as $N\to\infty$ for a sequence of
quasiuniform meshes.
Given a density function $w:\Sigma\to\RR$, our goal is to implement the
evaluation of the linear retarded layer potential of \cref{sec:rp_lc}
\begin{equation*}
  \opTk{k} w(\TX{x}) =
  \sum\nolimits_{\sigma\in\Sigma_N}\opTk{k}w\rvertPnl[\TX{x}] =
  \sum\nolimits_{\sigma\in\Sigma_N^\Xi(\TX{x})}\opTk{k}w\rvertPnl[\TX{x}],
\end{equation*}
where $w\rvertPnl$ is defined as in \cref{eq:ext_pnl_z}. In an approach we are
inclined to label ``naive'', $\Sigma_N^\Xi(\TX{x})$ is constructed by
considering each panel $\sigma\in\Sigma_N$ individually and verifying if
$\Xi(\TX{x})\cap\sigma$ is nonempty. Clearly, this procedure involves
$\mathcal{O}(N)$ operations (the maximum amount of operations necessary to
verify $\Xi(\TX{x})\cap\sigma\neq\varnothing$ is independent of $N$), spoiling
the $\mathcal{O}(N^{2/3})$ behavior dictated by the cardinality of
$\Sigma_N^\Xi(\TX{x})$.

In the subsequent paragraphs, we exhibit a straightforward algorithm for
constructing $\Sigma_N^\Xi(\TX{x})$ more efficiently. Assume we were given a set
$X_N(\TX{x})\subset\Sigma_N$ such that
$\Sigma_N^\Xi(\TX{x})\subset X_N(\TX{x})$ and
$|X_N(\TX{x})|\leq C|\Sigma_N^\Xi(\TX{x})|$ held for some $C\geq 1$ independent
of $N$. We proceed naively on $X_N(\TX{x})$ by checking every
$\sigma\in X_N(\TX{x})$ if $\Xi(\TX{x})\cap\sigma$ is nonempty and if so,
it is a member of $\Sigma_N^\Xi(\TX{x})$. By assumption
$|X_N(\TX{x})|<C|\Sigma_N^\Xi(\TX{x})|=\mathcal{O}(N^{2/3})$ holds, hence the
computational cost of this approach is dictated by the amount of operations
necessary to set up $X_N(\TX{x})$. The algorithm for assembling $X_N(\TX{x})$
is based on a hierarchical organization of the panels.
\begin{definition}[Binary Cluster Tree]\label{def:bct}
  Let $\IDX$ be an index set with $|\IDX|=N$. Each index in $\IDX$ corresponds
  to a unique panel in $\Sigma_N$ via the bijection $\IDX\to\Sigma_N$,
  $i\mapsto\sigma_i$. Let $\BCT\coloneqq(\VTX,\EDG)$ be a tree with vertex set
  $\VTX$, edge set $\EDG$, and let $n_{\min}\in\NN$ be given. For $v\in \VTX$
  define the sets $\sons(v)\coloneqq\{w\in\VTX:(v,w)\in\EDG\}$ and
  $\leaves(\BCT)\coloneqq\{w\in\VTX:\sons(v)=\varnothing\}$. The tree
  $\BCT$ is called binary cluster tree if
  \begin{enumerate}[(i)]
  \item $\operatorname{root}\BCT=\IDX$,
  \item for all $v\in\VTX$ it holds $v\subset\IDX$ and $v\neq\varnothing$,
  \item $v\in\leaves(\BCT) \Rightarrow |v|\leq n_{\min}$,
  \item for all $v\in\VTX$ it holds either $\sons(v)=\varnothing$ or
    $\sons(v)=\{v^\prime,v^{\prime\prime}\}$ with
    $v=v^\prime \,\cup\, v^{\prime\prime}$ and
    $v^\prime\,\cap\, v^{\prime\prime}=\varnothing$, i.e., any
    $v\not\in\leaves(\BCT)$ has two disjoint successors whose union is $v$.
  \end{enumerate}
  The vertices are called clusters and we identify $\VTX$ with $\BCT$, i.e., we
  write $v\in\BCT$ instead of $v\in\VTX$.
\end{definition}
The construction of the cluster tree $\BCT$ is performed as discussed in
\cite[Section 1.4.1.1, Equation (1.21)]{Bebendorf2008}, which involves
$\mathcal{O}(N\log(N))$ operations for quasiuniform meshes
\cite[Theorem 1.27]{Bebendorf2008}. In essence, $\BCT$ depends on $\Sigma_N$
and $n_{\min}$ only, therefore, it is set up once and used for every
evaluation point.
For each $v\in\BCT$ we find a bounding ball $B(v)\subset\RR[4]$, such
that $\Xi(\TX{x})\cap B(v)=\varnothing$ implies
$\sigma_i\not\in\Sigma_N^\Xi(\TX{x})$ for any $\sigma_i,i\in v$.

\begin{theorem}\label{thm:phi_xi_sphere}
  Let $\TX{x}\in\RR[4]$ be given and $B_r(\TX{y})$ be the closed ball of
  radius $r>0$ around $\TX{y}\in\RR[4]$. It holds
  \begin{equation*}
    \min_{\TX{z} \in B_r(\TX{y})} \phi_\Xi(\TX{x}-\TX{z})
    = \phi_\Xi(\TX{x}-\TX{y})-d_r(\norm{x-y})
    ,\quad
    \max_{\TX{z} \in B_r(\TX{y})} \phi_\Xi(\TX{x}-\TX{z})
    = \phi_\Xi(\TX{x}-\TX{y})+r\sqrt{2},
  \end{equation*}
  where $d_r:[0,\infty)\to\left[r,r\sqrt{2}\right]$ is defined by
  \begin{equation*}
    d_r:\alpha\mapsto\begin{cases} \alpha+\sqrt{r^2-\alpha^2} &
      \text{if}~0\leq\alpha<r/\sqrt{2},
      \\ r\sqrt{2} &\text{if}~\alpha\geq r/\sqrt{2}.\end{cases}
  \end{equation*}
\end{theorem}
\begin{proof}
  We employ the decomposition $\TX{z}\coloneqq(s,z)$ with $s\in\RR$ and
  $z\in\RR[3]$. It holds $\TX{z}\in B_r(\TX{y})$ iff
  $\norm{z-y}\leq\sqrt{r^2-(s-\tau)^2}$ holds. We expand
  \begin{equation}\label{eq:phi_exp}
    \phi_\Xi(\TX{x}-\TX{z}) = \norm{x-z}-(t-s) =\norm{x-y+y-z}-(t-\tau+\tau-s)
  \end{equation}
  and the triangle inequality yields
  \begin{equation*}
    \phi_\Xi(\TX{x}-\TX{z}) \leq \norm{x-y}+\norm{z-y} -(t-\tau)+(s-\tau)
    \leq \phi_\Xi(\TX{x}-\TX{y})+\sqrt{r^2-(s-\tau)^2}+(s-\tau)
    .
  \end{equation*}
  The maximum is attained at $s-\tau=r/\sqrt{2}$, yielding the bound
  $\phi_\Xi(\TX{x}-\TX{z}) \leq \phi_\Xi(\TX{x}-\TX{y})+r\sqrt{2}$ for any
  $\TX{z}\in B_r(\TX{y})$. The bound is sharp, because it holds
  \begin{equation*}
    \TX{z}\coloneqq \TX{y}- \frac{r}{\sqrt{2}}
    \begin{pmatrix}-1\\(x-y)/\norm{x-y}\end{pmatrix} \in
    B_r(\TX{y}) \quad \Rightarrow \quad
    \phi_\Xi(\TX{x}-\TX{z})=\phi_\Xi(\TX{x}-\TX{y})+r\sqrt{2}
    .
  \end{equation*}
  Considering the lower bound, we apply the reverse triangle inequality to
  \cref{eq:phi_exp}
  \begin{equation}\label{eq:phi_low_bnd}
    \phi_\Xi(\TX{x}-\TX{z}) \geq \left|\norm{x-y}-\norm{z-y}\right|
    -(t-\tau)+(s-\tau) \geq
    \left|\norm{x-y}-\norm{z-y}\right|-(t-\tau)-\sqrt{r^2-\|z-y\|^2},
  \end{equation}
  where we used $s-\tau\geq -|s-\tau|\geq -\sqrt{r^2-\|z-y\|^2}$. We abbreviate
  $\beta\coloneqq\norm{z-y}$ and declare $f_r:[0,r]\to\RR$ by
  $\beta\mapsto \left|\norm{x-y}-\beta\right|-\sqrt{r^2-\beta^2}$ such that
  $\phi_\Xi(\TX{x}-\TX{z}) \geq f_r(\beta)-(t-\tau)$ holds. The (weak)
  derivative
  \begin{equation*}
    f_r^\prime : \beta\mapsto -\sgn\left(\norm{x-y}-\beta\right)
    +\beta/\sqrt{r^2-\beta^2}
  \end{equation*}
  satisfies $f_r^\prime(\beta)<0$ iff both $\beta<\norm{x-y}$ and
  $\beta\leq r/\sqrt{2}$ hold. We distinguish two scenarios:
  $\norm{x-y}\geq r/\sqrt{2}$ or $\norm{x-y}< r/\sqrt{2}$.
  Assuming $\norm{x-y}\geq r/\sqrt{2}$ holds, it follows $f_r^\prime(\beta)<0$
  iff $\beta\leq r/\sqrt{2}$ and $f_r$ is monotonically decreasing in
  $[0,r/\sqrt{2})$ while it is nondecreasing (almost) everywhere else.
  Therefore, $f_r$ attains its minimum at $\beta=r/\sqrt{2}$. Insertion of
  $\norm{z-y}=r/\sqrt{2}\leq\norm{x-y}$ in \cref{eq:phi_low_bnd} yields
  $\phi_\Xi(\TX{x}-\TX{z}) \geq \phi_\Xi(\TX{x}-\TX{y})-r\sqrt{2}$ for any
  $\TX{z}\in B_r(\TX{y})$. The sharpness of this bound is confirmed by
  \begin{equation*}
    \TX{z}\coloneqq \TX{y}+ \frac{r}{\sqrt{2}}
    \begin{pmatrix}-1\\(x-y)/\norm{x-y}\end{pmatrix} \in
    B_r(\TX{y}) \quad \Rightarrow \quad
     \phi_\Xi(\TX{x}-\TX{z}) =
    \norm{x-y-\frac{r}{\sqrt{2}}\frac{x-y}{\norm{x-y}}}-(t-\tau+r/\sqrt{2}),
  \end{equation*}
  which yields $\phi_\Xi(\TX{x}-\TX{z})=\phi_\Xi(\TX{x}-\TX{y})-r\sqrt{2}$ for
  $\norm{x-y}\geq r/\sqrt{2}$. We turn our attention to the case
  $\norm{x-y}<r/\sqrt{2}$. Since $f_r^\prime(\beta)<0$ holds iff
  $\beta<\norm{x-y}$ holds, the minimum of $f_r$ is located at
  $\beta=\norm{x-y}$. Insertion of $\beta=\norm{z-y}=\norm{x-y}$ in
  \cref{eq:phi_low_bnd} yields
  \begin{equation*}
    \phi_\Xi(\TX{x}-\TX{z}) \geq -(t-\tau)-\sqrt{r^2-\|x-y\|^2} =
    \phi_\Xi(\TX{x}-\TX{y})-\norm{x-y}-\sqrt{r^2-\|x-y\|^2}.
  \end{equation*}
  Finally
  \begin{equation*}
    \TX{z}\coloneqq 
    \begin{pmatrix}\tau-(r^2-\norm{x-y}^2)^{1/2}\\x\end{pmatrix} \in
    B_r(\TX{y}) \quad \Rightarrow \quad
    \phi_\Xi(\TX{x}-\TX{z}) = -(t-\tau)-\sqrt{r^2-\|x-y\|^2}
  \end{equation*}
  confirms the sharpness of the stated bound for $\norm{x-y}< r/\sqrt{2}$ and
  the proof is complete.
\end{proof}
\Cref{thm:phi_xi_sphere} is applied in line 2 of \cref{alg:lit_leaves}. There is
no root of $\TX{z}\mapsto\phi_\Xi(\TX{x}-\TX{z})$ for $\TX{z}\in B_r(\TX{y})$,
i.e., the bounding ball of $v\in\BCT$ is not lit by $\Xi(\TX{x})$, iff either
its minimum value is positive or its maximum value is negative. As an
initialization, set $L(\TX{x})\coloneqq\varnothing$ and call
\texttt{ApproximateLitLeaves}$(L(\TX{x}),\operatorname{root}\BCT)$. Once the
algorithm concludes, set
$X_N(\TX{x})\coloneqq\bigcup_{v\in L(\TX{x})}\bigcup_{i\in v}\sigma_i$. In
line $1$ of \cref{alg:lit_leaves}, the routine \texttt{GetBoundingSphere}($v$)
returns a precomputed bounding sphere that encloses all $\sigma_i,i\in v$. In
our implementation, we use a slightly modified version of the algorithm laid
out in \cite{Ritter1990}, which computes a nonminimal bounding sphere.
Although the cited source exhibits the algorithm explicitly for $\RR[3]$, its
extension to $\RR[n],n\in\NN$ is obvious.

\begin{algorithm}
  \caption{\texttt{ApproximateLitLeaves}$(L(\TX{x}),v)$. Given the current
    iterate of $L(\TX{x})\subset\leaves\BCT$ and $v\in\BCT$.}
  \label{alg:lit_leaves}
  \begin{algorithmic}[1]
    \State $B_r(\TX{y})\coloneqq$ \texttt{GetBoundingSphere}($v$)
    \If{$\phi_\Xi(\TX{x}-\TX{y})<-r\sqrt{2}$ \textbf{or}
      $\phi_\Xi(\TX{x}-\TX{y})>d_r(\norm{x-y})$}
    \State \textbf{return}
    \EndIf
    \If{$v\in\leaves(\BCT)$}
    \State $L(\TX{x}) \gets L(\TX{x})\cup\{v\}$
    \Else
    \ForAll {$v^\prime\in\sons(v)$}
    \State \texttt{ApproximateLitLeaves}$(L(\TX{x}),v^\prime)$
    \EndFor
    \EndIf
  \end{algorithmic}
\end{algorithm}

\begin{remark}\label{rem:fast_bem}
  The proposed algorithm is based on concepts typically encountered in
  fast BEMs. Nevertheless, this approach does not constitute a traditional
  ``fast method''; it implements evaluation procedures of exact (apart from
  quadrature) retarded potential integral operators efficiently. The necessity
  for such implementational tricks, even outside the realm of fast methods,
  arises because retarded potentials are not classically global operators, but
  their integrals are supported on (subsets of) the hypersurface $\Xi(\TX{x})$.
\end{remark}

\section{Numerical experiments}\label{sec:numex}

The purpose of this section is to verify the proposed methods and provide
evidence about the capacity of space-time BEMs for RPBIEs. Further
numerical experiments are given in \cite[Chapter 4]{Poelz2021}.

\subsection{Experiment 1: computation of lit panels}\label{sec:ex_clst}

The first experiment investigates the performance of the method discussed in
\cref{sec:algo}, which computes the set of lit panels $\Sigma_N^\Xi(\TX{x})$
efficiently. Two computational domains are considered, namely the unit cube
$\Omega^-=\left(-\frac{1}{2},\frac{1}{2}\right)^3$ with
$T=1$ and the unit ball $\Omega^-=\left\{x\in\RR[3]:\norm{x}<1\right\}$ with
$T=5$. We examine the three evaluation points $\TX{x}_\ast\coloneqq(T,x_\ast)$,
$\ast\in\{A,B,C\}$, where $x_\ast\in\RR[3]$ are the spatial components
\begin{equation*}
  x_A\coloneqq\begin{pmatrix}0&0&0\end{pmatrix}^\top
  ,\quad
  x_B\coloneqq\tfrac{1}{\sqrt{3}}\begin{pmatrix}1&1&1\end{pmatrix}^\top
  ,\quad
  x_C\coloneqq\begin{pmatrix}-1&-\frac{\sqrt{2}}{2}&-\frac{1}{\pi}
  \end{pmatrix}^\top
  .
\end{equation*}
As a preliminary test, the cardinalities of the set of lit panels
$|\Sigma_N^\Xi(\TX{x})|$ and the proxy set $|X_N(\TX{x})|$ are investigated. The
evaluation point $\TX{x}_C$ is chosen and three upper bounds for the size of
leaf-level clusters $n_{\min}\in\{1,5,50\}$ are employed. Results of this study
are displayed in \cref{fig:ex_set_card}. On the one hand, the conjectured
$\mathcal{O}(N^{2/3})$ behavior of $|\Sigma_N^\Xi(\TX{x}_C)|$ can be observed.
On the other hand, the results suggest the existence of a constant $C>1$ such
that $|X_N(\TX{x}_C)|<C|\Sigma_N^\Xi(\TX{x}_C)|$ holds, which is the key
assumption in \cref{sec:algo}. Furthermore,
$|\Sigma_N^\Xi(\TX{x}_C)|<|X_N(\TX{x}_C)|$ holds in all considered cases, even
for $n_{\min}=1$.

\begin{figure}[htbp]
  \centering
  \subcaptionbox{cardinalities of sets, unit cube%
    \label{fig:ex_set_card_cube}}%
  {\includegraphics[height=55mm]{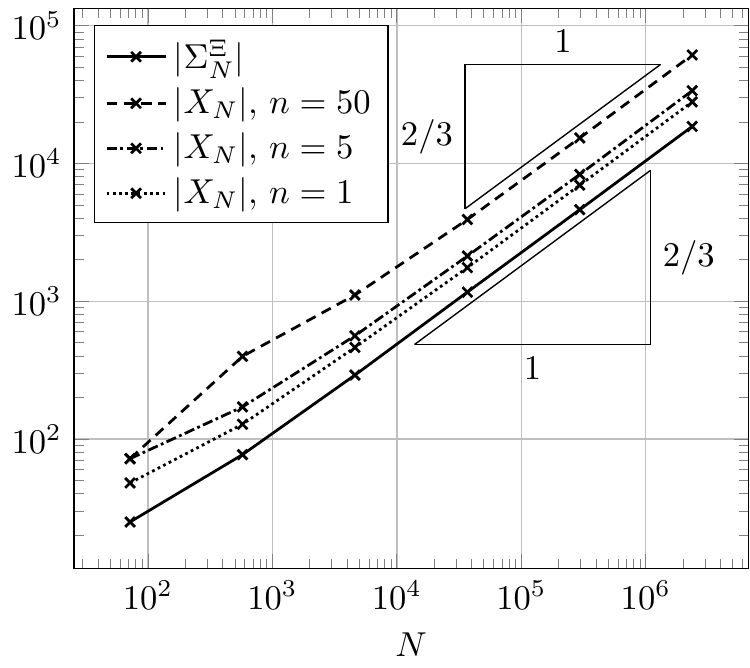}}
  \hspace{15mm}
  \subcaptionbox{cardinalities of sets, unit sphere%
    \label{fig:ex_set_card_sphere}}%
  {\includegraphics[height=55mm]{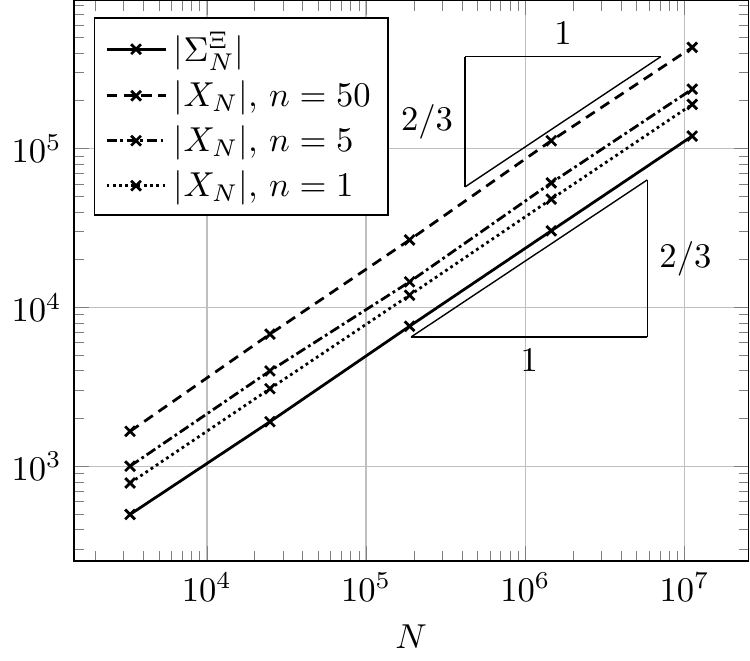}}
  \caption{Cardinalities of the set of lit panels $\Sigma_N^\Xi(\TX{x}_C)$ and
    the proxy set $X_N(\TX{x}_C)$; the number $n\coloneqq n_{\min}\in\{1,5,50\}$
    indicates the maximum size of leaf-level clusters.}
  \label{fig:ex_set_card}
\end{figure}

A second example is considered, which aims at demonstrating the increase in
performance achieved by the proposed technique. The naive approach (check every
$\sigma\in\Sigma_N$ if $\sigma\in\Sigma_N^\Xi(\TX{x})$ holds) is compared to the
procedure outlined in \cref{sec:algo}:
\begin{enumerate}
\item $L(\TX{x})\coloneqq\varnothing$,
  \texttt{ApproximateLitLeaves}$(L(\TX{x}),\operatorname{root}\BCT)$,
\item $X_N(\TX{x})\coloneqq\bigcup_{v\in L(\TX{x})}\bigcup_{i\in v}\sigma_i$,
\item check every $\sigma\in X_N(\TX{x})$ if $\sigma\in \Sigma_N^\Xi(\TX{x})$
  holds.
\end{enumerate}
Both approaches construct the same set $\Sigma_N^\Xi(\TX{x})$, however, the
elapsed times differ. Let $t_N$ be the execution time of the naive approach
and $t_n$ be the time required to perform above list of three steps. Again, the
subscript $n\coloneqq n_{\min}\in\{1,50\}$ is the maximum size of leaf-level
clusters. All execution times (provided in ordinary time, seconds) reported in
\cref{fig:ex_comp_time} are minimum values of five consecutive runs of the
stated procedures. The displayed results suggest that $t_N$ behaves like
$\mathcal{O}(N)$, while $t_n$ features an $\mathcal{O}(N^{2/3})$ behavior. The
three
solid lines in \cref{fig:ex_comp_time_cube,fig:ex_comp_time_sphere} overlap,
indicating that $t_N$ depends little on the actual position of $\TX{x}$.
However, $t_n$ depends heavily on $\TX{x}$, at least in
\cref{fig:ex_comp_time_cube}. This shows that the proposed algorithm can
yield particularly large reductions of the execution time for points
$\TX{x}$ such that $|\Sigma_N^\Xi(\TX{x})|$ is small. The difference between
$t_{50}$ and $t_1$ is noteworthy, suggesting that the extreme choice
$n_{\min}=1$ is advantageous if $\Sigma_N^\Xi(\TX{x})$ is computed for
sufficiently many evaluation points $\TX{x}$.

\begin{figure}[htbp]
  \centering
  \subcaptionbox{construction time of $\Sigma_N^\Xi$, unit cube%
    \label{fig:ex_comp_time_cube}}%
  {\includegraphics[height=55mm]{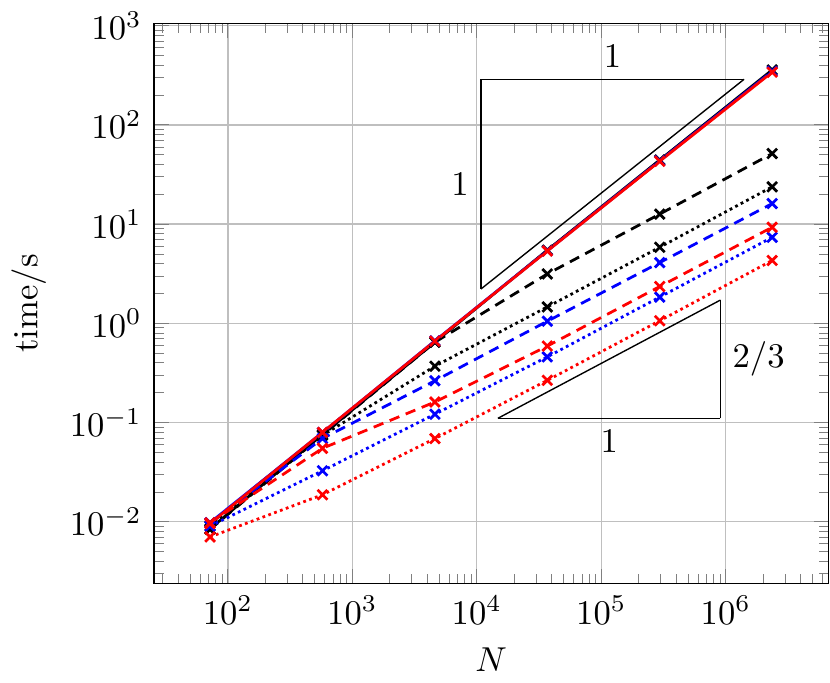}}
  \hfill
  \subcaptionbox{construction time of $\Sigma_N^\Xi$, unit sphere%
    \label{fig:ex_comp_time_sphere}}%
  {\includegraphics[height=55mm]{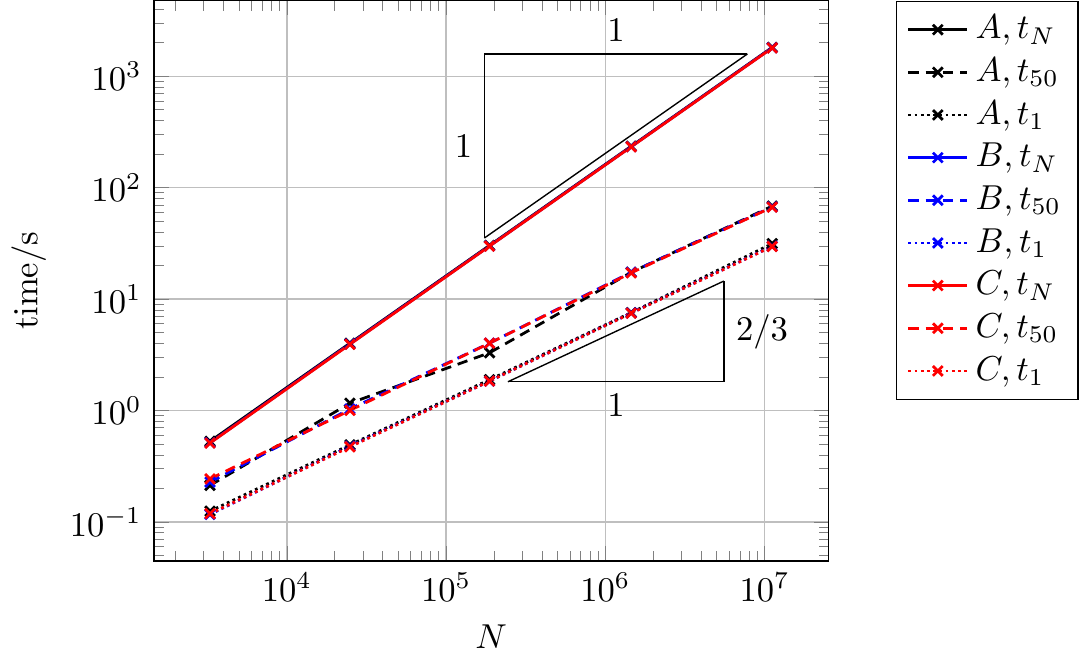}}
  \caption{Construction time of the set of lit panels $\Sigma_N^\Xi(\TX{x}_\ast)$
    for $\ast\in\{A,B,C\}$; the ordinates depict the elapsed time in seconds.
    The legend on the right is valid for both plots and the subscript
    $n\coloneqq n_{\min}\in\{1,50\}$ of $t_n$ indicates the corresponding
    maximum size of leaf-level clusters. The value $t_N$ is the execution time
    of the naive approach (check every $\sigma\in\Sigma_N$). In both plots,
    the lines of $t_N$ overlap for all points $\TX{x}_A$, $\TX{x}_B$, and
    $\TX{x}_A$.%
    }
  \label{fig:ex_comp_time}
\end{figure}

Finally, note that all exhibited elapsed times are obtained on an
\textup{Intel\textsuperscript{\textregistered} Core\texttrademark}
i7-8700 desktop machine with a clock speed of \SI{3.2}{\GHz}.
The absolute values of the execution times in \cref{fig:ex_comp_time} are of
little significance because the implementation is both single-threaded and
immature. Nevertheless, the presented results reveal the improvement
of the asymptotic behavior due to the method proposed in \cref{sec:algo}.

\subsection{Experiment 2: quadrature method}\label{sec:ex_int}

The goal of this section is to verify the quadrature scheme of
\cref{sec:q_inner}. To this end, we revisit the experiments carried out in
\cite[Section 4.1]{Poelz2019a} and compare the results. The experimental
setup is recapped for the sake of completeness. As a computational domain we
consider the unit cube $\Omega^-=\left(-\frac{1}{2},\frac{1}{2}\right)^3$
with $T=5$. Given a function $u:Q^+\to\RR$ subject to \cref{eq:pde,eq:ic}, we
evaluate the function
\begin{equation}\label{eq:u_khoff}
  \widetilde{u}:\TX{x}\mapsto\begin{cases}
    \potDl\trDiri[+] u(\TX{x})-\potSl\trNeum[+] u(\TX{x})
    &\text{if}~\TX{x}\in Q^+,
    \\
    \bioDl\trDiri[+] u(\TX{x})-\bioSl\trNeum[+] u(\TX{x})
    +\mathcal{J}(x)\trDiri[+] u(\TX{x})
    &\text{if}~\TX{x}\in\Sigma,
  \end{cases}
\end{equation}
where $\mathcal{J}:\Gamma\to[0,1]$ is the solid angle, see
\cite[Equation (6.11)]{Steinbach2008}. All integral operators in
\cref{eq:u_khoff} are approximated by the quadrature method introduced in
\cref{sec:q_inner}. This is the only relevant source of the error
$u-\widetilde{u}$ because $u=\widetilde{u}$ would hold if all integral
operators were evaluated exactly (note that the exact Cauchy data are used in
\cref{eq:u_khoff}). The chosen solution $u$ of \cref{eq:pde,eq:ic} is a
spherical wave function
\begin{equation}\label{eq:sphere_wave}
  u:\TX{x}\mapsto\frac{\mu(t-\norm{x-y_S})}{\norm{x-y_S}}
  ,
\end{equation}
where $y_S\in\Omega^-$ is set to
$y_S\coloneqq \begin{pmatrix}-0.1&-0.2&-0.3\end{pmatrix}^\top$ and
$\mu:\RR\to\RR$ is given by
\begin{equation*}
  \mu:t\mapsto\begin{cases}
    \exp\left(\left(\frac{t^2}{4}-t\right)^{-1}\right)&\text{if}~t\in (0,4),
    \\
    0&\text{otherwise}
    .
  \end{cases}
\end{equation*}
This choice of $\mu$ is smooth $\mu\in C^\infty(\RR)$ and causal, i.e.,
$\mu(t)=0$ holds for all $t\leq 0$.

In the first example, a mesh $\Sigma_N$ of $N=180$ panels is considered. The
evaluation point is set to $\TX{x}_d\coloneqq (T,x_d)$, where $x_d\in\RR[3]$
is given by $x_d\coloneqq\begin{pmatrix}0.5&0.5&0.5\end{pmatrix}^\top + d
\begin{pmatrix}1&0&0\end{pmatrix}^\top$ with $0\leq d=\sDist[\Gamma](x_d)$.
The relative error measure
\begin{equation*}
  \mathrm{e}_d\coloneqq
  \frac{|u(\TX{x}_d)-\widetilde{u}(\TX{x}_d)|}{|u(\TX{x}_d)|}
\end{equation*}
is evaluated for $d\in\{0,0.1,1,3\}$. The quadrature scheme discussed at the
end of \cref{sec:q_inner} has two main input parameters, namely the number of
quadrature points per direction $n_G\in\NN$ and the depth of the quadtree
$r_{\max}\in\NN_0$. We consider $r_{\max}\in\{10,20\}$ and study the convergence
with respect to $n_G$. Results of this experiment are exhibited in
\cref{fig:ex_q_pnt_r1} for $r_{\max}=10$ and in \cref{fig:ex_q_pnt_r2} for
$r_{\max}=20$. Clearly, $\mathrm{e}_d$ decays rapidly as $n_G$ is increased.
However, convergence ceases once the error falls below a certain threshold,
which depends on $r_{\max}$. The existence of such a threshold suggests that
certain quadtree cells fit no admissible scenario even after $r_{\max}$ steps
of subdivision. These cells are treated by low-order approximations and,
therefore, convergence with respect to $n_G$ is capped. Nevertheless,
for $r_{\max}=20$ the achievable error $\mathrm{e}_d\approx\num{1E-12}$ is
already rather close to machine epsilon. It is noteworthy that the case $d=0$,
which involves weakly singular kernel functions, is handled just as well as the
cases with $d>0$. This behavior is due to the employed transformations, which
regularize the integrand, see \cref{lem:kernel_para}. Finally, it is emphasized
that the results for $r_{\max}=20$ are quite comparable to the data provided in
our earlier work \cite[Figure 3(a)]{Poelz2019a}.

We consider a further test in order to support the capacity of the quadrature
scheme for weakly singular integral kernels. As in
\cite[Section 4.1]{Poelz2019a} a different mesh of the computational domain
$\Sigma_N$, consisting of $N=288$ panels, is employed. The examined relative
error measure is given by
\begin{equation*}
  \mathrm{e}_\Sigma\coloneqq \frac{\sum_{i=1}^N|u(\TX{x}_i)-
    \widetilde{u}(\TX{x}_i)|}{\sum_{i=1}^N|u(\TX{x}_i)|}
  ,
\end{equation*}
where $\{\TX{x}_i\}_{i=1}^N$ is the set of centroids of the panels in
$\Sigma_N$. Results of this convergence study are displayed in
\cref{fig:ex_q_avg_cube} for $r_{\max}\in\{7,14\}$. Again, we observe that
$\mathrm{e}_\Sigma$ decays swiftly as $n_G$ is increased, until it falls below
a certain magnitude, which depends on $r_{\max}$. For $r_{\max}=7$ convergence
ceases at $\mathrm{e}_\Sigma\approx\num{1E-5}$, while $r_{\max}=14$ leads to an
error threshold more than three orders of magnitude smaller. Overall, the
results depicted in \cref{fig:ex_q_avg_cube} are quite similar to the ones
displayed in \cite[Figure 3(b)]{Poelz2019a}. We conclude that the quadrature
approach laid out in \cref{sec:q_inner} is, albeit immature, indeed capable of
computing highly accurate pointwise evaluations of retarded layer potentials in
the space-time setting.

\begin{figure}[htbp]
  \centering
  \subcaptionbox{pointwise error $\mathrm{e}_d$ for $r_{\max}=10$%
    \label{fig:ex_q_pnt_r1}}%
  {\includegraphics[height=42mm]{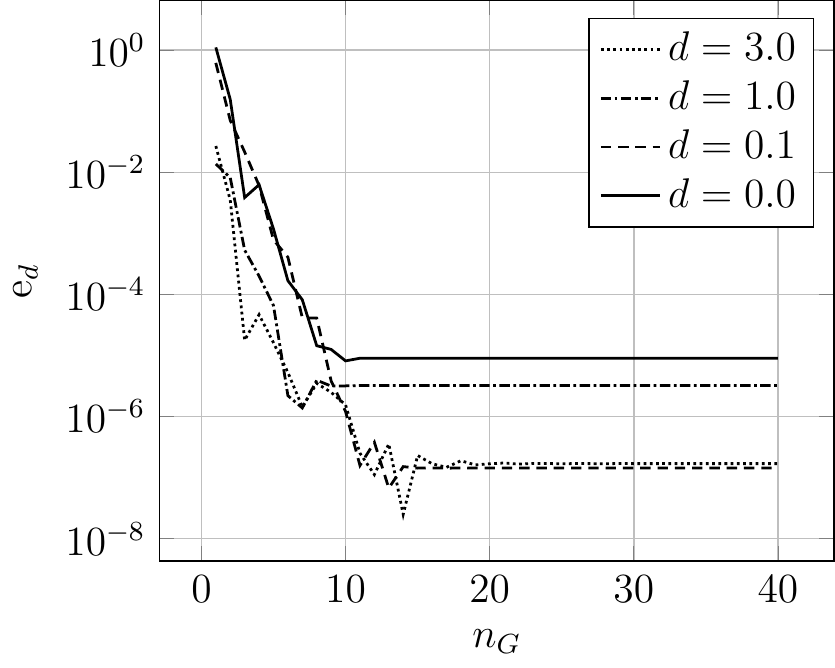}}
  \hfill
  \subcaptionbox{pointwise error $\mathrm{e}_d$ for $r_{\max}=20$%
    \label{fig:ex_q_pnt_r2}}%
  {\includegraphics[height=42mm]{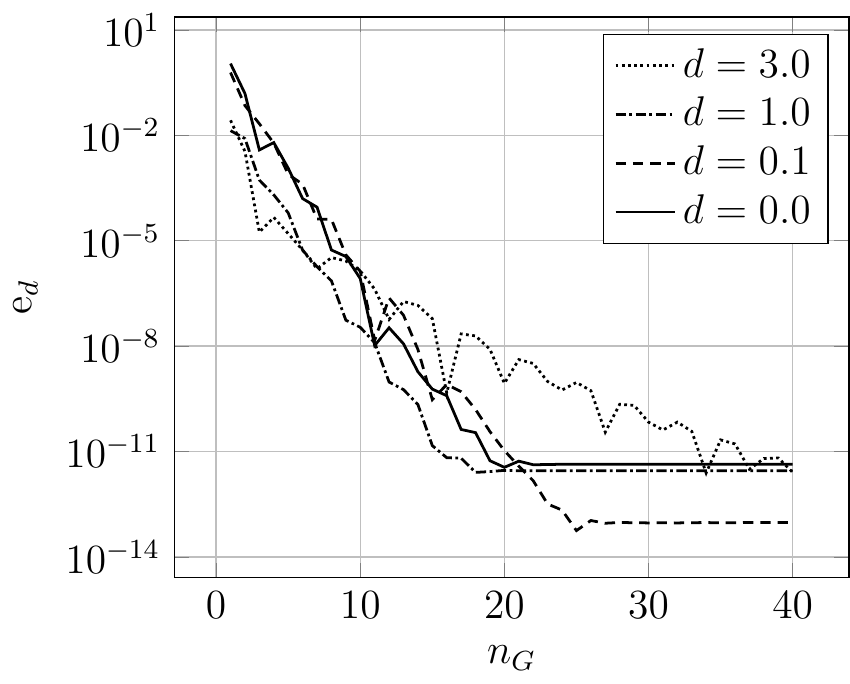}}
  \hfill
  \subcaptionbox{error $\mathrm{e}_\Sigma$ averaged over 288 points%
    \label{fig:ex_q_avg_cube}}%
  {\includegraphics[height=42mm]{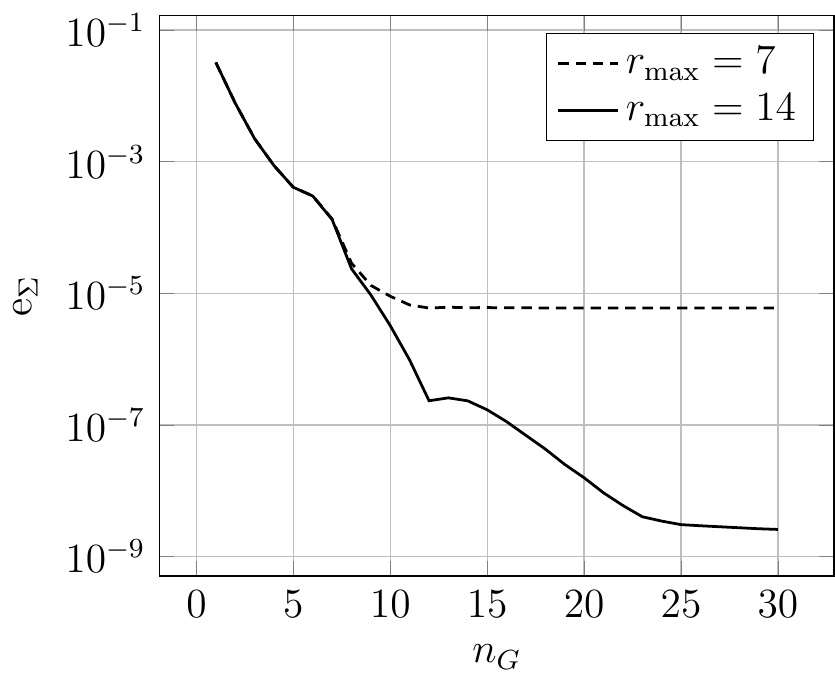}}
  \caption{Convergence study of the quadrature scheme discussed in
    \cref{sec:q_inner}; $r_{\max}$ and $n_G$ denote the quadtree depth and the
    number of quadrature points per direction, respectively. The total
    number of quadrature points behaves like $\mathcal{O}(n_G^2)$.}
  \label{fig:ex_q_pnt}
\end{figure}

\subsection{Experiment 3: space-time BEMs}\label{sec:ex_bem}

The final experiment is intended to verify the space-time BEMs discussed
in \cref{sec:disc} and to illustrate their capacity. In all following tests,
the parameters for the inner quadrature are set to
$(r_{\max},n_G)\coloneqq (7,8)$, while the outer quadrature employs
$m_Q\coloneqq 3$ points per direction, see \cref{fig:int_tri_mid}.

The first test investigates the indirect BEM \cref{eq:df_ind} and we employ the
exact solutions of $\bioSl w=g$ derived in \cite{Sauter2014} for spherical
scatterers $\Gamma=\nSphere$. Denote by $Y_n^m:\nSphere\to\CC$ the spherical
harmonic function of degree $n\in\NN_0$ and order $m\in\ZZ$ such that
$-n\leq m\leq n$ holds. Let $g:\Sigma\to\RR$ be defined by
$g:\TX{x}\mapsto g_0(t)\Re\left(Y_1^0(x)\right)$, where $\Re$ denotes the real
part and $g_0:\RR\to\RR$ reads
\begin{equation*}
  g_0:t\mapsto\begin{cases}
    t^4\exp(-2t)&\text{if}~t>0,
    \\
    0&\text{if}~t\leq 0.
  \end{cases}
\end{equation*}
In this case, the solution $w:\Sigma\to\RR$ of $\bioSl w=g$ is given by
$w:\TX{x}\mapsto w_0(t)\Re\left(Y_1^0(x)\right)$, where $w_0$ is provided
in \cite[Equation (4.18)]{Sauter2014}.
We solve \cref{eq:df_ind} for $w_h\in S_h^0(\Sigma_N)$ and evaluate the
error measures
\begin{equation}\label{eq:err_l2}
  \mathrm{e}_{\text{abs}}\coloneqq \norm{w-w_h}[\LTSig]
  ,\quad
  \mathrm{e}_{\text{BEM}}\coloneqq \frac{\norm{w-w_h}[\LTSig]}{\norm{w}[\LTSig]}
  ,\quad
  \mathrm{e}_{\text{opt}}\coloneqq\min_{z_h\in S_h^0(\Sigma_N)}
  \frac{\norm{w-z_h}[\LTSig]}{\norm{w}[\LTSig]}
  .
\end{equation}
Note that the minimum in $\mathrm{e}_{\text{opt}}$ is attained by the
$\LTSig$-orthogonal projection of $w$ onto $S_h^0(\Sigma_N)$. A convergence
study is displayed in \cref{fig:ex_slp_ind}. Both $\mathrm{e}_{\text{opt}}$ and
$\mathrm{e}_{\text{BEM}}$ exhibit first-order convergence with respect to $h$.
Therefore, the BEM approximation seems to satisfy a quasioptimality principle
in $\LTSig$, or in other words, there seems to exist a $T$-dependent constant
$C(T)>1$ such that $\mathrm{e}_{\text{BEM}}\leq C(T) \mathrm{e}_{\text{opt}}$
holds.

\begin{figure}[htbp]
  \centering
  \subcaptionbox{relative error measures of \cref{eq:err_l2}%
    \label{fig:ex_slp_ind_rel}}%
  {\includegraphics[width=70mm]{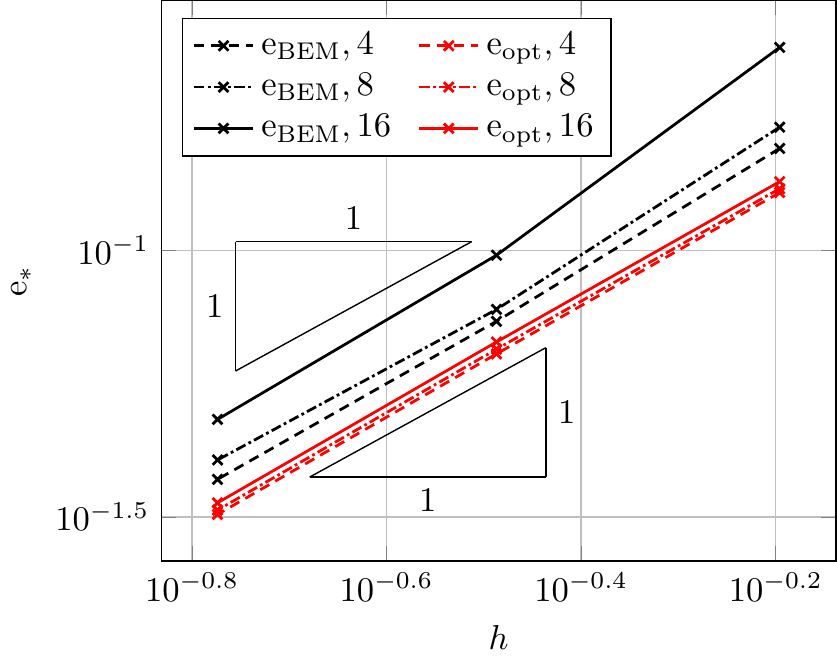}}
  \hfill
  \subcaptionbox{absolute error measure of \cref{eq:err_l2}%
    \label{fig:ex_slp_ind_abs}}%
  {\includegraphics[width=70mm]{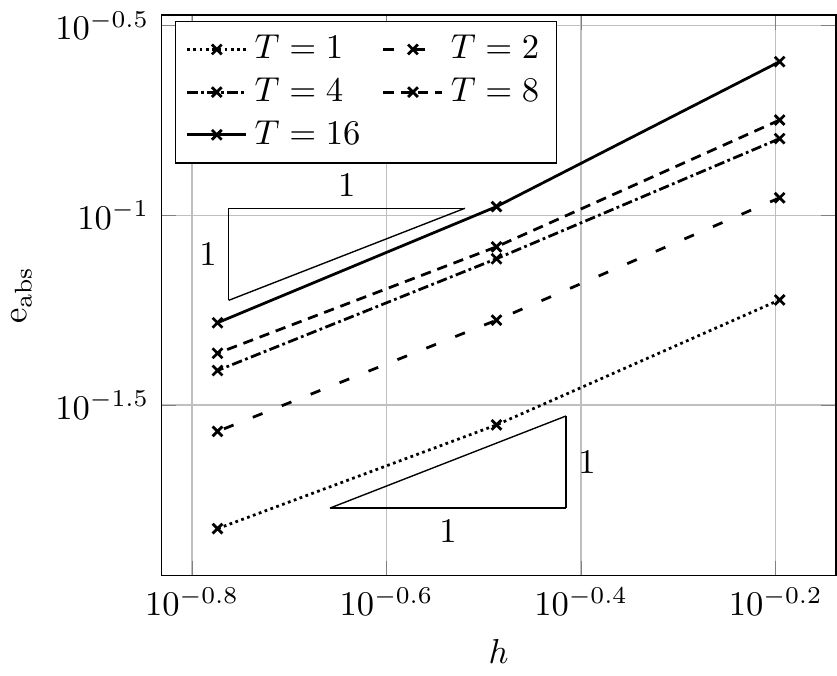}}
  \caption{Convergence study of the indirect BEM \cref{eq:df_ind} on the
    unit sphere for various simulation end times $T$. The numbers in the
    legend of \cref{fig:ex_slp_ind_rel} represent the corresponding values of
    $T\in\{4,8,16\}$.}
  \label{fig:ex_slp_ind}
\end{figure}

The final example investigates the performance of the direct BEM
\cref{eq:df_dir}. The computational domain is set to
$\Omega^-=\left(-\frac{1}{2},\frac{1}{2}\right)^3$ and the employed reference
solution is given by \cref{eq:sphere_wave}, where $\mu\in C^2(\RR)$ reads
\begin{equation*}
  \mu:t\mapsto\begin{cases}
    t^3\exp(-t)&\text{if}~t>0,
    \\
    0&\text{if}~t\leq 0.
  \end{cases}
\end{equation*}
The discretized RPBIE \cref{eq:df_dir} is solved for $w_h\approx\trNeum[+]u$
and the relative error measures of \cref{eq:err_l2} are computed. Additionally,
the error in the wave field $u-u_h$ is studied, where
$u_h\coloneqq\potDl\opQ_h^1 g-\potSl w_h$ is given by the discretized
Kirchhoff's formula. We consider $26$ evaluation points
$\TX{x}_i\coloneqq(T,x_i)$ for $i=1,\dots,26$, where each $x_i\in\RR[3]$
lies on the boundary of the cube $\left(-\frac{3}{5},\frac{3}{5}\right)^3$.
The following relative error measure is reported
\begin{equation}\label{eq:err_int_q}
  \mathrm{e}_Q\coloneqq\frac{1}{26}\sum_{i=1}^{26}
  \frac{|u(\TX{x}_i)-u_h(\TX{x}_i)|}{|u(\TX{x}_i)|}
  .
\end{equation}

\begin{figure}[htbp]
  \centering
  \subcaptionbox{$\LTSig$-norm of the error of the Neumann unknown%
    \label{fig:ex_slp_dir_neum}}%
  {\includegraphics[height=55mm]{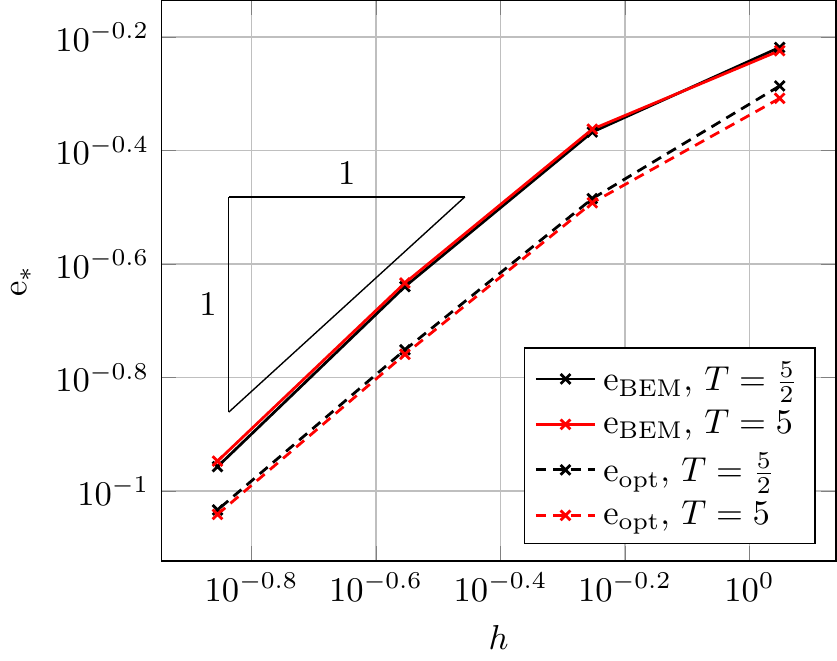}}
  \hspace{15mm}
  \subcaptionbox{error of wave field in $Q^+$ at 26 points \cref{eq:err_int_q}%
    \label{fig:ex_slp_dir_int}}%
  {\includegraphics[height=55mm]{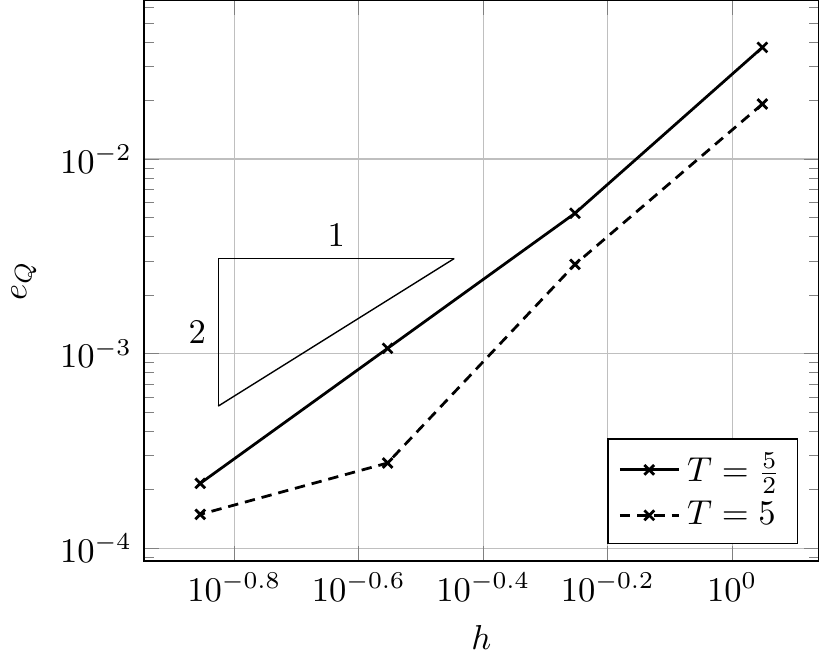}}
  \caption{Convergence study of the direct BEM \cref{eq:df_dir} on the
    unit cube for two different simulation end times $T$.}
  \label{fig:ex_slp_dir}
\end{figure}

\Cref{fig:ex_slp_dir} exhibits the results of the conducted convergence study.
Similar to the previous test, \cref{fig:ex_slp_dir_neum} displays an
$\mathcal{O}(h)$ behavior of $\mathrm{e}_{\text{BEM}}$ and
$\mathrm{e}_{\text{opt}}$. This provides further evidence that the BEM
solution seems to satisfy a quasioptimality principle. \Cref{fig:ex_slp_dir_int}
indicates that the examined mesh sizes still lie in the preasymptotic regime of
$\mathrm{e}_Q$. On average, we observe that the pointwise error in the wave
field converges quadratically with respect to $h$. If the theory of BIEs for
elliptic problems indeed carried over to hyperbolic RPBIEs, we could indeed
expect second-order convergence, see \cite[Equation (12.21)]{Steinbach2008}.

In both tests the proposed space-time Galerkin BEMs for RPBIEs yield optimal
convergence rates in the $\LTSig$-norm. The results are tremendously better
than the ones obtained by the space-time collocation BEM we developed in
\cite{Poelz2019a}. Furthermore, the provided evidence confirms
that the quadrature scheme for space-time bilinear forms discussed in
\cref{sec:q_outer} yields sufficiently accurate matrix entries such that the
overall convergence of the BEM solution is not spoiled (even for few
quadrature points $m_Q=3$).

\section{Conclusion}\label{sec:conc}

This paper presents a discretization scheme for variational integral equations
of the wave equation based on space-time boundary elements. We derive an
integral formula for retarded layer potentials which fits the space-time
setting exceptionally well. A carefully constructed parametrization of the
light cone simplifies these integrals greatly for piecewise flat boundary
meshes. This enables the application of existing quadrature schemes developed
for implicitly defined domains. Since retarded layer potentials induce
non-smooth functions, the evaluation of related space-time bilinear forms is
accomplished via a tentative low-order approach. Nevertheless, numerical
evidence suggests that the proposed methods provide sufficiently accurate
evaluations of retarded layer potentials and associated Galerkin matrix
entries. In all examined tests, the error of the Galerkin approximation
converges quasioptimally in the $\LTSig$-norm. Furthermore, the efficient
computation of the set of panels lit by the light cone is addressed. The
proposed algorithm is based on a hierarchical structure of the mesh and
facilitates computations of the set of lit panels in (nigh) optimal complexity.
Although several ideas and techniques described in this work are still in an
early stage of development, numerical experiments indicate their potential.

\appendix

\section{An example of singularities in retarded layer potentials}
\label{app:appendix}

In this example, we study the behavior of the function
$\TX{x}\mapsto\potSl w(\TX{x})$ for a simple configuration, similar to the
analysis in \cite{Stephan2008}. Let $w$ be the indicator function of the
tetrahedron $\sigma$ defined by
\begin{equation*}
  \sigma \coloneqq \operatorname{conv}\{he_1,~he_2,~he_3,~0e_4\},
\end{equation*}
where $\operatorname{conv}$ denotes the convex hull, $\{e_i\}_{i=1}^4$ is the
canonical basis of $\RR[4]$, and $h>0$. The normal vector of $\sigma$ is given
by $\nu^\top=\begin{pmatrix}0&0&0&1\end{pmatrix}$ and the signed distance
functions of the bounding half-spaces $\phi_i:\RR[4]\to\RR,i=1,\dots,4$ read
\begin{equation}\label{eq:app_ilf_face}
  \phi_1:\TX{y}\mapsto -\tau ,\quad
  \phi_2:\TX{y}\mapsto -y_2 ,\quad
  \phi_3:\TX{y}\mapsto -y_1 ,\quad
  \phi_4:\TX{y}\mapsto \left(\tau+y_1+y_2-h\right)/\sqrt{3}
  .
\end{equation}
Let $\mathcal{I}$ be as in \cref{def:int_inner} with
$k:(x,y)\mapsto\norm{x-y}^{-1}$ as in \cref{eq:kfunc}, apart from the factor
$4\pi$. The computation of the integral follows along the lines of
\cref{sec:q_inner}. Application of \cref{eq:kernel_para_full} leads to
\begin{equation}\label{eq:example_int_slp}
  \mathcal{I} = \intOp{\psi_{\TX{x}}^{-1}(\Xi(\TX{x})\cap\sigma)}%
  {\frac{r_0\rho\sin\theta}{\sqrt{\cos^2\theta+\rho^2\sin^2\theta}}}{S(\zeta)}.
\end{equation}
\Cref{def:choice_r} yields the components of the parametrization $R=I$ and
$r_0=t$. For simplicity, we choose evaluation points
$\TX{x}\in\mathcal{T}_\sigma$ only, which is equivalent to $x_3=0$.
This choice implies $\rho_0=0$ in \cref{eq:para_hypl} and we deduce that any
$\zeta\coloneqq(\rho,\varphi,\theta)\in\mathcal{P}$ with $\rho\neq 0$ satisfies
$\psi_{\TX{x}}(\zeta)\in\mathcal{T}_\sigma$ iff $\theta=\pi/2$. Insertion of
$\theta=\pi/2$ in \cref{eq:example_int_slp} yields
\begin{equation*}
  \mathcal{I} = r_0 \intOp{U}{}{(\rho,\varphi)}, \quad
  U\coloneqq\{(\rho,\varphi)\in[0,1]\times[0,2\pi):
  \phi_{\sigma}\circ\psi_{\TX{x}}\circ\ell_2(\rho,\varphi)<0\},
\end{equation*}
where $\psi_{\TX{x}}\circ\ell_2:(\rho,\varphi)\mapsto\TX{x}-r_0\rho
\begin{pmatrix}1& \cos\varphi& \sin\varphi& 0\end{pmatrix}^\top$ holds. This
leads to
\begin{align*}
  \phi_1\circ\psi_{\TX{x}}\circ\ell_2 :~&
  (\rho,\varphi)\mapsto r_0\rho-t,\\
  \phi_2\circ\psi_{\TX{x}}\circ\ell_2:~&
  (\rho,\varphi)\mapsto r_0\rho\sin\varphi-x_2,\\
  \phi_3\circ\psi_{\TX{x}}\circ\ell_2 :~&
  (\rho,\varphi)\mapsto r_0\rho\cos\varphi-x_1.
\end{align*}
We exclude $\phi_4$ because of the following considerations. From
\cref{eq:dist_faces} it follows that $U$ is the set in which all four
signed distance functions are negative
\begin{equation*}
  U\coloneqq\{(\rho,\varphi)\in[0,1]\times[0,2\pi):\phi_{i}\circ\psi_{\TX{x}}
  \circ\ell_2(\rho,\varphi)<0\quad\forall i=1,\dots,4\}.
\end{equation*}
Define the set in which the first three functions are negative
\begin{equation}\label{eq:cond_param_in_tet}
  U_0\coloneqq\{(\rho,\varphi)\in[0,1]\times[0,2\pi):
  r_0\rho<t ~~\text{and}~~ r_0\rho\sin\varphi<x_2
  ~~\text{and}~~r_0\rho\cos\varphi<x_1\}.
\end{equation}
We have $U=U_0$ iff $\phi_4\circ\psi_{\TX{x}}\circ\ell_2(\rho,\varphi)<0$ holds
for all $(\rho,\varphi)\in U_0$. From \cref{eq:app_ilf_face} it follows that
$U=U_0$ can be guaranteed by choosing $h$ sufficiently large. From here on,
assume that $U=U_0$ holds and we are left with computing
\begin{equation*}
  \mathcal{I} = r_0 \intOp{U_0}{}{(\rho,\varphi)}.
\end{equation*}
We consider two cases: the light cone approaches either a corner or an edge of
the tetrahedron $\sigma$.
\par\smallskip

\textbf{Case 1: Corner.} Let $0<\varepsilon<t$ and $\TX{x}\coloneqq
\begin{pmatrix}t& -\varepsilon\sqrt{2}/2& -\varepsilon\sqrt{2}/2& 0
\end{pmatrix}^\top$. In this case, \cref{eq:cond_param_in_tet} becomes
\begin{equation*}
  r_0\rho<t ~\text{and}~ r_0\rho\sin\varphi < -\varepsilon\sqrt{2}/2
  ~\text{and}~ r_0\rho\cos\varphi < -\varepsilon\sqrt{2}/2
  .
\end{equation*}
Since $r_0\rho\geq 0$ holds, the latter two conditions can be true only if
$\varphi\in(\pi,3\pi/2)$. This leads to the maps
\begin{equation*}
  \rho_1 : (\pi,5\pi/4]\to\RR,~ \varphi\mapsto
  -\frac{\sqrt{2}}{2}\frac{\varepsilon}{r_0}\frac{1}{\sin\varphi}
  ,\quad
  \rho_2 : (5\pi/4,3\pi/2)\to\RR,~ \varphi\mapsto
  -\frac{\sqrt{2}}{2}\frac{\varepsilon}{r_0}\frac{1}{\cos\varphi}
  ,
\end{equation*}
where the symmetry $\rho_1(5\pi/4-\varphi)=\rho_2(5\pi/4+\varphi)$ is evident.
The angle $\varphi_1\in(\pi,5\pi/4)$ such that $\rho_1(\varphi_1)=t/r_0$ holds
is given by $\varphi_1=\pi+\arcsin\left(\sqrt{2}\varepsilon/(2t)\right)$.
Exploiting the symmetry about $5\pi/4$ yields
\begin{align*}
  \mathcal{I}(t,\varepsilon) &= 2 r_0 \int_{\varphi_1}^{5\pi/4}
  \int_{-\varepsilon\sqrt{2}/(2t\sin\varphi)}^1 \mathrm{d}\rho\mathrm{d}\varphi\\ &=
  2 t \left( \frac{\pi}{4}-
    \arcsin\left(\frac{\sqrt{2}}{2}\frac{\varepsilon}{t}\right) \right) -
  \sqrt{2}\varepsilon\log\left(\sqrt{2}-1\right) +
  \sqrt{2}\varepsilon\log\left(
    \sqrt{2}\frac{1-\sqrt{1-\varepsilon^2/(4t^2)}}{\varepsilon/t}\right)
  .
\end{align*}
Its partial derivatives are
\begin{align*}
  \partial_t \mathcal{I}(t,\varepsilon) &=
  2 \left(\pi/4-\arcsin\left(\sqrt{2}\varepsilon/(2t)\right) \right)
  \\
  \partial_\varepsilon \mathcal{I}(t,\varepsilon) &=
  -\sqrt{2}\log\left(\sqrt{2}-1\right) + \sqrt{2}\log\left(
    \sqrt{2}\frac{1-\sqrt{1-\varepsilon^2/(4t^2)}}{\varepsilon/t}\right)
  .
\end{align*}
The singularity at $\varepsilon\to 0$ occurs as $\TX{x}$ approaches the
boundary of $\sigma$. The second-order derivatives are
\begin{equation*}
  \partial_t^2 \mathcal{I}(t,\varepsilon) =
  \frac{2\varepsilon}{t^2\sqrt{2-\varepsilon^2/t^2}}
  ,\quad
  \partial_{t\varepsilon} \mathcal{I}(t,\varepsilon) =
  -\frac{2}{t\sqrt{2-\varepsilon^2/t^2}}
  ,\quad
  \partial_\varepsilon^2 \mathcal{I}(t,\varepsilon) =
  \frac{2}{\varepsilon\sqrt{2-\varepsilon^2/(4t^2)}}
  .
\end{equation*}
Apart from the obvious singularity for $\varepsilon\to 0$ the behavior for
$\varepsilon\to\sqrt{2}t$ is not relevant, since $\varepsilon>t$ implies
$\mathcal{I}(t,\varepsilon)=0$. From $t\to 0$ it follows $\varepsilon\to 0$
(recall $0<\varepsilon<t$) and we conclude that singularities (up to
second-order derivatives) are confined to $\partial\sigma$.
\par\smallskip

\textbf{Case 2: Edge.} Let $0<\varepsilon<t$ and $\TX{x}\coloneqq
\begin{pmatrix}t& -\varepsilon& h/2& 0\end{pmatrix}^\top$. In this case,
\cref{eq:cond_param_in_tet} is equivalent to
\begin{equation*}
  r_0\rho<t ~\text{and}~ r_0\rho\sin\varphi < h/2
  ~\text{and}~ r_0\rho\cos\varphi < -\varepsilon
  .
\end{equation*}
The middle condition is trivially satisfied for sufficiently large $h$, while
$r_0\rho\geq 0$ implies that the latter holds only if
$\varphi\in(\pi/2,3\pi/2)$. We employ the parametrization
$\rho_3:(\pi/2,3\pi/2)\to\RR$, $\varphi\mapsto-\varepsilon/(r_0\cos\varphi)$.
The angle $\varphi_2\in(\pi/2,\pi)$ with $\rho_3(\varphi_2)=t/r_0$ is given by
$\varphi_2=\arccos\left(-\varepsilon/t\right)$. Exploiting the symmetry
$\rho_3(\pi+\varphi)=\rho_3(\pi-\varphi)$ yields
\begin{equation*}
  \mathcal{I}(t,\varepsilon) = 2 r_0 \int_{\varphi_2}^{\pi}
  \int_{-\varepsilon/(t\cos\varphi)}^1 \mathrm{d}\rho\mathrm{d}\varphi =
  2 t \left( \pi-\arccos\left(-\frac{\varepsilon}{t}\right) \right) -
  2\varepsilon\log\left(
    \frac{1+\sqrt{1-\varepsilon^2/t^2}}{\varepsilon/t}\right)
\end{equation*}
whose first-order partial derivatives are
\begin{equation*}
  \partial_t \mathcal{I}(t,\varepsilon) =
  2 \left( \pi-\arccos\left(-\frac{\varepsilon}{t}\right) \right)
  ,\quad
  \partial_\varepsilon \mathcal{I}(t,\varepsilon) =
  -2\log\left(\frac{1+\sqrt{1-\varepsilon^2/t^2}}{\varepsilon/t}\right)
  .
\end{equation*}
As in the first case, the singularity at $\varepsilon\to 0$ occurs as
$\TX{x}$ approaches $\partial\sigma$. The second-order partial derivatives,
however, reveal a more intriguing behavior
\begin{equation*}
  \partial_t^2 \mathcal{I}(t,\varepsilon) = 2\frac{\varepsilon}{t^2}
  \frac{1}{\sqrt{1-\varepsilon^2/t^2}}
  ,\quad
  \partial_{t\varepsilon} \mathcal{I}(t,\varepsilon) = -\frac{2}{t}
  \frac{1}{\sqrt{1-\varepsilon^2/t^2}}
  ,\quad
  \partial_\varepsilon^2 \mathcal{I}(t,\varepsilon) = \frac{2}{\varepsilon}
  \frac{1}{\sqrt{1-\varepsilon^2/t^2}}
  .
\end{equation*}
The singularity at $\varepsilon\to t$ does not occur in the first case. It is
indeed relevant because the light cone barely grazes the edge of $\sigma$ as
$\varepsilon\to t$. This singular behavior is not limited to $\partial\sigma$
but ``propagates'' on the line $\varepsilon=t$. For $\varepsilon=t$ it holds
$\TX{x} = \begin{pmatrix}0&0&h/2&0\end{pmatrix}^\top+
t\begin{pmatrix}1&-1&0&0\end{pmatrix}^\top$ and we are inclined to induce that
the observed singularity is related to forward light cones with apexes at the
edges of $\sigma$, cf. \cite{Stephan2008}.

\section*{Acknowledgments}
The first author is gratefully indebted to Daniel Sch\"ollhammer for his
advice and assistance in carrying out the numerical experiments of
\cref{sec:ex_bem} on adequate computers.

This research did not receive any specific grant from funding agencies in the
public, commercial, or not-for-profit sectors.

%% file: mainArxiv.bbl
\begin{thebibliography}{10}

\bibitem{Neumueller2011}
M.~Neum{\"u}ller and O.~Steinbach, ``Refinement of flexible space-time finite
  element meshes and discontinuous {G}alerkin methods,'' {\em Comput. Vis.
  Sci.}, vol.~14, no.~5, pp.~189--205, 2011.

\bibitem{Steinbach2015}
O.~Steinbach, ``Space-time finite element methods for parabolic problems,''
  {\em Comput. Methods Appl. Math.}, vol.~15, no.~4, pp.~551--566, 2015.

\bibitem{Gopalakrishnan2017a}
J.~Gopalakrishnan, J.~Sch{\"o}berl, and C.~Wintersteiger, ``Mapped tent
  pitching schemes for hyperbolic systems,'' {\em SIAM J. Sci. Comput.},
  vol.~39, no.~6, pp.~B1043--B1063, 2017.

\bibitem{Gopalakrishnan2019}
J.~Gopalakrishnan and P.~Sep{\'u}lveda, ``A space-time {DPG} method for the
  wave equation in multiple dimensions,'' in {\em Space-Time Methods}
  (U.~Langer and O.~Steinbach, eds.), vol.~25 of {\em Radon Series on
  Computational and Applied Mathematics}, ch.~4, pp.~117--140, Berlin, Boston:
  De Gruyter, 2019.

\bibitem{Wang2015}
L.~Wang and P.-O. Persson, ``A high-order discontinuous {G}alerkin method with
  unstructured space-time meshes for two-dimensional compressible flows on
  domains with large deformations,'' {\em Comput. \& Fluids}, vol.~118,
  pp.~53--68, 2015.

\bibitem{Langer2016}
U.~Langer, S.~Moore, and M.~Neum{\"u}ller, ``Space-time isogeometric analysis
  of parabolic evolution problems,'' {\em Comput. Methods Appl. Mech. Engrg.},
  vol.~306, pp.~342--363, 2016.

\bibitem{Doerfler2016}
W.~D{\"o}rfler, S.~Findeisen, and C.~Wieners, ``Space-time discontinuous
  {G}alerkin discretizations for linear first-order hyperbolic evolution
  systems,'' {\em Comput. Methods Appl. Math.}, vol.~16, no.~3, pp.~409--428,
  2016.

\bibitem{Poelz2019}
D.~P\"{o}lz, M.~Gfrerer, and M.~Schanz, ``Wave propagation in elastic trusses:
  An approach via retarded potentials,'' {\em Wave Motion}, vol.~87,
  pp.~37--57, 2019.

\bibitem{Gander2016}
M.~Gander and M.~Neum{\"u}ller, ``Analysis of a new space-time parallel
  multigrid algorithm for parabolic problems,'' {\em SIAM J. Sci. Comput.},
  vol.~38, no.~4, pp.~A2173--A2208, 2016.

\bibitem{Neumueller2013}
M.~Neum{\"u}ller, {\em Space-Time Methods: Fast Solvers and Applications},
  vol.~20 of {\em Monographic Series TU Graz: Computation in Engineering and
  Science}.
\newblock Verlag der Technischen Universit{\"a}t Graz, 2013.

\bibitem{Gopalakrishnan2015}
J.~Gopalakrishnan, P.~Monk, and P.~Sep\'{u}lveda, ``A tent pitching scheme
  motivated by {F}riedrichs theory,'' {\em Comput. Math. Appl.}, vol.~70,
  no.~5, pp.~1114--1135, 2015.

\bibitem{Perugia2020}
I.~Perugia, J.~Sch\"{o}berl, P.~Stocker, and C.~Wintersteiger, ``Tent pitching
  and {T}refftz-{DG} method for the acoustic wave equation,'' {\em Comput.
  Math. Appl.}, vol.~79, no.~10, pp.~2987--3000, 2020.

\bibitem{Uengoer2002}
A.~\"{U}ng\"{o}r and A.~Sheffer, ``Pitching tents in space-time: Mesh
  generation for discontinuous {G}alerkin method,'' {\em Internat. J. Found.
  Comput. Sci.}, vol.~13, no.~02, pp.~201--221, 2002.

\bibitem{Abboud2011}
T.~Abboud, P.~Joly, J.~Rodr\'{i}guez, and I.~Terrasse, ``Coupling discontinuous
  {G}alerkin methods and retarded potentials for transient wave propagation on
  unbounded domains,'' {\em J. Comput. Phys.}, vol.~230, no.~15,
  pp.~5877--5907, 2011.

\bibitem{Costabel2017}
M.~Costabel and F.-J. Sayas, ``Time-dependent problems with the boundary
  integral equation method,'' {\em Encycl. Comput. Mech. Second Ed.}, vol.~2,
  pp.~1--24, 2017.

\bibitem{Ha-Duong2003}
T.~Ha-Duong, B.~Ludwig, and I.~Terrasse, ``A {G}alerkin {BEM} for transient
  acoustic scattering by an absorbing obstacle,'' {\em Internat. J. Numer.
  Methods Engrg.}, vol.~57, no.~13, pp.~1845--1882, 2003.

\bibitem{Davies2004}
P.~Davies and D.~Duncan, ``Stability and convergence of collocation schemes for
  retarded potential integral equations,'' {\em SIAM J. Numer. Anal.}, vol.~42,
  no.~3, pp.~1167--1188, 2004.

\bibitem{Sauter2013}
S.~Sauter and A.~Veit, ``A {G}alerkin method for retarded boundary integral
  equations with smooth and compactly supported temporal basis functions,''
  {\em Numer. Math.}, vol.~123, no.~1, pp.~145--176, 2013.

\bibitem{Gimperlein2018}
H.~Gimperlein, F.~Meyer, C.~{\"O}zdemir, D.~Stark, and E.~Stephan, ``Boundary
  elements with mesh refinements for the wave equation,'' {\em Numer. Math.},
  vol.~139, no.~4, pp.~867--912, 2018.

\bibitem{Frangi2000}
A.~Frangi, ````{C}ausal'' shape functions in the time domain boundary element
  method,'' {\em Comput. Mech.}, vol.~25, pp.~533--541, Jun 2000.

\bibitem{Manson2019}
N.~Manson and J.~Tausch, ``Quadrature for parabolic {G}alerkin {BEM} with
  moving surfaces,'' {\em Comput. Math. Appl.}, vol.~77, no.~1, pp.~1--14,
  2019.

\bibitem{Tausch2019}
J.~Tausch, ``{N}ystr{\"o}m method for {BEM} of the heat equation with moving
  boundaries,'' {\em Adv. Comput. Math.}, vol.~45, pp.~2953--2968, Dec 2019.

\bibitem{Poelz2019a}
D.~P\"{o}lz and M.~Schanz, ``Space-time discretized retarded potential boundary
  integral operators: Quadrature for collocation methods,'' {\em SIAM J. Sci.
  Comput.}, vol.~41, no.~6, pp.~A3860--A3886, 2019.

\bibitem{Joly2017}
P.~Joly and J.~Rodr{\'i}guez, ``Mathematical aspects of variational boundary
  integral equations for time dependent wave propagation,'' {\em J. Integral
  Equations Appl.}, vol.~29, no.~1, pp.~137--187, 2017.

\bibitem{Banz2016}
L.~Banz, H.~Gimperlein, Z.~Nezhi, and E.~Stephan, ``Time domain {BEM} for sound
  radiation of tires,'' {\em Comput. Mech.}, vol.~58, no.~1, pp.~45--57, 2016.

\bibitem{Veit2016}
A.~Veit, M.~Merta, J.~Zapletal, and D.~Luk\'{a}\v{s}, ``Efficient solution of
  time-domain boundary integral equations arising in sound-hard scattering,''
  {\em Internat. J. Numer. Methods Engrg.}, vol.~107, no.~5, pp.~430--449,
  2016.

\bibitem{Aimi2009}
A.~Aimi, M.~Diligenti, C.~Guardasoni, I.~Mazzieri, and S.~Panizzi, ``An energy
  approach to space-time {G}alerkin {BEM} for wave propagation problems,'' {\em
  Internat. J. Numer. Methods Engrg.}, vol.~80, no.~9, pp.~1196--1240, 2009.

\bibitem{Aimi2010}
A.~Aimi, M.~Diligenti, and S.~Panizzi, ``Energetic {G}alerkin {BEM} for wave
  propagation {N}eumann exterior problems,'' {\em CMES-Comp. Model. Eng.},
  vol.~58, no.~2, pp.~185--219, 2010.

\bibitem{Erichsen1998}
S.~Erichsen and S.~Sauter, ``Efficient automatic quadrature in 3-d {G}alerkin
  {BEM},'' {\em Comput. Methods Appl. Mech. Engrg.}, vol.~157, no.~3,
  pp.~215--224, 1998.

\bibitem{Sauter2011}
S.~Sauter and C.~Schwab, {\em Boundary Element Methods}, vol.~39 of {\em
  Springer Series in Computational Mathematics}.
\newblock Springer Berlin Heidelberg, 2011.

\bibitem{Aimi2013}
A.~Aimi, M.~Diligenti, A.~Frangi, and C.~Guardasoni, ``{N}eumann exterior wave
  propagation problems: computational aspects of 3d energetic {G}alerkin
  {BEM},'' {\em Comput. Mech.}, vol.~51, no.~4, pp.~475--493, 2013.

\bibitem{Gimperlein2018a}
H.~Gimperlein and D.~Stark, ``Algorithmic aspects of enriched time domain
  boundary element methods,'' {\em Eng. Anal. Bound. Elem.}, vol.~100,
  pp.~118--124, 2019.

\bibitem{McLean2000}
W.~McLean, {\em Strongly Elliptic Systems and Boundary Integral Equations}.
\newblock Cambridge University Press, 2000.

\bibitem{Ortner1980a}
N.~Ortner, ``{R}egularisierte {F}altung von {D}istributionen. {T}eil 2: Eine
  {T}abelle von {F}undamentall{\"o}sungen,'' {\em Z. Angew. Math. Phys.},
  vol.~31, no.~1, pp.~155--173, 1980.

\bibitem{Sayas2016}
F.-J. Sayas, {\em Retarded Potentials and Time Domain Boundary Integral
  Equations: A Road Map}, vol.~50 of {\em Springer Series in Computational
  Mathematics}.
\newblock Cham: Springer, 2016.

\bibitem{Bamberger1986}
A.~Bamberger and T.~Ha~Duong, ``Formulation variationnelle espace-temps pour le
  calcul par potentiel retard\'{e} de la diffraction d'une onde acoustique
  ({I}),'' {\em Math. Methods Appl. Sci.}, vol.~8, no.~1, pp.~405--435, 1986.

\bibitem{Aimi2012}
A.~Aimi, M.~Diligenti, A.~Frangi, and C.~Guardasoni, ``A stable 3d energetic
  {G}alerkin {BEM} approach for wave propagation interior problems,'' {\em Eng.
  Anal. Bound. Elem.}, vol.~36, no.~12, pp.~1756--1765, 2012.

\bibitem{Federer1996}
H.~Federer, {\em Geometric Measure Theory}.
\newblock Springer Berlin Heidelberg, 1996.

\bibitem{Hoermander2003}
L.~H{\"o}rmander, {\em The Analysis of Linear Partial Differential Operators
  {I}: Distribution Theory and Fourier Analysis}.
\newblock Springer Berlin Heidelberg, 2003.

\bibitem{Poelz2021}
D.~P\"{o}lz, {\em Space-Time Boundary Elements for Retarded Potential Integral
  Equations}, vol.~41 of {\em Monographic Series TU Graz: Computation in
  Engineering and Science}.
\newblock Verlag der Technischen Universit\"{a}t Graz, 2021.

\bibitem{Ciarlet2002}
P.~Ciarlet, {\em The Finite Element Method for Elliptic Problems}, vol.~40 of
  {\em Classics in Applied Mathematics}.
\newblock Philadelphia: SIAM, 2002.

\bibitem{Karabelas2015}
E.~Karabelas and M.~Neum{\"u}ller, ``Generating admissible space-time meshes
  for moving domains in $d+1$-dimensions,'' 2015.

\bibitem{Steinbach2008}
O.~Steinbach, {\em Numerical Approximation Methods for Elliptic Boundary Value
  Problems: Finite and Boundary Elements}.
\newblock Springer Science \& Business Media, 2008.

\bibitem{Mueller2013}
B.~M\"{u}ller, F.~Kummer, and M.~Oberlack, ``Highly accurate surface and volume
  integration on implicit domains by means of moment-fitting,'' {\em Internat.
  J. Numer. Methods Engrg.}, vol.~96, no.~8, pp.~512--528, 2013.

\bibitem{Saye2015}
R.~Saye, ``High-order quadrature methods for implicitly defined surfaces and
  volumes in hyperrectangles,'' {\em SIAM J. Sci. Comput.}, vol.~37, no.~2,
  pp.~A993--A1019, 2015.

\bibitem{Fries2017}
T.~Fries, S.~Omerovi\'{c}, D.~Sch\"{o}llhammer, and J.~Steidl, ``Higher-order
  meshing of implicit geometries-part {I}: Integration and interpolation in cut
  elements,'' {\em Comput. Methods Appl. Mech. Engrg.}, vol.~313, pp.~759--784,
  2017.

\bibitem{Fries2016}
T.-P. Fries and S.~Omerovi\'{c}, ``Higher-order accurate integration of
  implicit geometries,'' {\em Internat. J. Numer. Methods Engrg.}, vol.~106,
  no.~5, pp.~323--371, 2016.

\bibitem{Lehrenfeld2016}
C.~Lehrenfeld, ``High order unfitted finite element methods on level set
  domains using isoparametric mappings,'' {\em Comput. Methods Appl. Mech.
  Engrg.}, vol.~300, pp.~716--733, 2016.

\bibitem{Gfrerer2018}
M.~Gfrerer and M.~Schanz, ``A high-order {FEM} with exact geometry description
  for the {L}aplacian on implicitly defined surfaces,'' {\em Internat. J.
  Numer. Methods Engrg.}, vol.~114, no.~11, pp.~1163--1178, 2018.

\bibitem{Stephan2008}
E.~Stephan, M.~Maischak, and E.~Ostermann, ``Transient boundary element method
  and numerical evaluation of retarded potentials,'' in {\em Computational
  Science - ICCS 2008} (M.~Bubak, G.~van Albada, J.~Dongarra, and P.~Sloot,
  eds.), pp.~321--330, Springer Berlin Heidelberg, 2008.

\bibitem{Ostermann2010}
E.~Ostermann, {\em Numerical Methods for Space-Time Variational Formulations of
  Retarded Potential Boundary Integral Equations}.
\newblock PhD thesis, Gottfried Wilhelm Leibniz Universit\"{a}t Hannover, 2010.

\bibitem{Bebendorf2008}
M.~Bebendorf, {\em Hierarchical Matrices}.
\newblock Springer, 2008.

\bibitem{Ritter1990}
J.~Ritter, ``An efficient bounding sphere,'' in {\em Graphics Gems}
  (A.~Glassner, ed.), pp.~301--303, San Diego, CA: Academic Press Inc., 1990.

\bibitem{Sauter2014}
S.~Sauter and A.~Veit, ``Retarded boundary integral equations on the sphere:
  exact and numerical solution,'' {\em IMA J. Numer. Anal.}, vol.~34, no.~2,
  pp.~675--699, 2014.

\end{thebibliography}
